\documentclass{amsart}
\usepackage{hyperref}
\usepackage[latin1]{inputenc}
\usepackage{graphicx}
\usepackage{amsmath}
\usepackage{amsthm}
\usepackage{amssymb,bbm}%
\usepackage[numbers, square]{natbib}

\newcommand{\bA}{\mathbb{A}}

\newcommand{\bN}{\mathbb{N}}
\newcommand{\bC}{\mathbb{C}}

\newcommand{\bR}{\mathbb{R}}

\newcommand{\bD}{\mathbb{D}}

\newcommand{\bV}{\mathbb{\hspace{-.07mm}V\hspace{-.2mm}}}

\newcommand{\cE}{\mathcal{E}}
\newcommand{\cF}{\mathcal{F}}
\newcommand{\cG}{\mathcal{G}}

\newcommand{\cL}{\mathcal{L}}
\newcommand{\cM}{\mathcal{M}}

\newcommand{\Beta}{\text{\rm Beta}}

\newcommand{\IPL}{\text{\rm IPL}}
\newcommand{\EPL}{\text{\rm EPL}}

\newcommand{\Normal}{\mathcal{N}}
\newcommand{\Exp}{\mathrm{Exp}}

\newcommand{\E}{\mathbb E}

\newcommand{\R}{\mathbb{R}}
\newcommand{\N}{\mathbb{N}}

\newcommand{\C}{\mathbb{C}}

\renewcommand{\P}{\mathbb{P}}

\renewcommand{\Re}{\operatorname{Re}}
\renewcommand{\Im}{\operatorname{Im}}

\newcommand{\Var}{\mathop{\mathrm{Var}}\nolimits}

\newcommand{\eps}{\varepsilon}
\newcommand{\eqdistr}{\stackrel{d}{=}}

\newcommand{\todistr}{\overset{d}{\underset{n\to\infty}\longrightarrow}}

\newcommand{\tomix}{\overset{mix}{\underset{n\to\infty}\longrightarrow}}
\newcommand{\toweak}{\overset{w}{\underset{n\to\infty}\longrightarrow}}
\newcommand{\toprobab}{\overset{P}{\underset{n\to\infty}\longrightarrow}}
\newcommand{\tofd}{\overset{f.d.d.}{\underset{n\to\infty}\longrightarrow}}
\newcommand{\toas}{\overset{a.s.}{\underset{n\to\infty}\longrightarrow}}
\newcommand{\toasw}{\overset{a.s.w.}{\underset{n\to\infty}\longrightarrow}}
\newcommand{\ton}{\overset{}{\underset{n\to\infty}\longrightarrow}}

\newcommand{\bsl}{\backslash}
\newcommand{\ind}{\mathbbm{1}}

\newcommand{\dd}{{\rm d}}
\newcommand{\eee}{{\rm e}}

\theoremstyle{plain}
\newtheorem{theorem}{Theorem}[section]
\newtheorem{lemma}[theorem]{Lemma}
\newtheorem{corollary}[theorem]{Corollary}
\newtheorem{proposition}[theorem]{Proposition}

\theoremstyle{definition}

\newtheorem{example}[theorem]{Example}

\theoremstyle{remark}
\newtheorem{remark}[theorem]{Remark}

\begin{document}

\author{Rudolf Gr\"ubel}
\address{Rudolf Gr\"ubel, Institut f\"ur Mathematische Stochastik,
Leibniz Universit\"at Hannover,
Welfengarten 1,
30167 Hannover, Germany
}
\email{rgrubel@stochastik.uni-hannover.de}

\author{Zakhar Kabluchko}
\address{Zakhar Kabluchko, Institut f\"ur Stochastik,
Universit\"at Ulm,
Helmholtzstr.\ 18,
89069 Ulm, Germany}
\email{zakhar.kabluchko@uni-ulm.de}

\title[Functional CLT for Branching Random Walks]{A functional central limit theorem for branching random walks, almost sure weak convergence, and applications to random trees}

\thanks{}
\begin{abstract}
Let $W_{\infty}(\beta)$ be the limit of the Biggins martingale $W_n(\beta)$ associated to a supercritical branching random walk with mean number of offspring $m$. We prove a functional central limit theorem stating that as $n\to\infty$ the process
$$
D_n(u):= m^{\frac 12 n} \left(W_{\infty}\left(\frac{u}{\sqrt n}\right) - W_{n}\left(\frac{u}{\sqrt n}\right) \right)
$$
converges weakly, on a suitable space of analytic functions, to a Gaussian random analytic function with random variance.
Using this result we prove central limit theorems for the total path length of random trees. In the setting of binary search trees, we recover a recent result of R.\ Neininger [Refined Quicksort Asymptotics, \textit{Rand.\ Struct.\ and Alg.}, to appear], but we also prove a similar theorem for uniform random recursive trees.
Moreover, we replace weak convergence in Neininger's theorem by the almost sure weak (a.s.w.)\ convergence of probability transition kernels.
In the case of binary search trees, our result states that
$$
\cL\left\{\sqrt{\frac{n}{2\log n}} \left(\EPL_{\infty} - \frac{\EPL_n-2n\log n}{n}\right)\Bigg | \cG_{n}\right\} \toasw \{\omega\mapsto\Normal_{0,1}\},
$$
where $\EPL_n$ is the external path length of a binary search tree $X_n$ with $n$ vertices, $\EPL_{\infty}$ is the limit of the R\'egnier martingale, and $\cL(\,\cdot\, |\cG_n)$ denotes the conditional distribution w.r.t.\ the $\sigma$-algebra $\cG_n$ generated by $X_1,\ldots,X_n$.
A.s.w.\ convergence is stronger than weak and even stable convergence. We prove several basic properties of the a.s.w.\ convergence
and study a number of further examples in which the a.s.w.\ convergence appears naturally. These include the classical central limit theorem for Galton--Watson processes and the P\'olya urn.
\end{abstract}

\maketitle

\section{Introduction}
The research that led to the present paper was motivated by a question from
the analysis of algorithms, specifically of the famous \textsc{Quicksort} and the closely
related binary search tree (BST) algorithms. The question concerns the second-order
(distributional) asymptotics of the number of comparisons needed by \textsc{Quicksort} or,
equivalently, of the total path length of the associated random binary search trees, if the
input to the algorithm is random.

Let the input sequence consist of independent random variables $U_1,U_2,\ldots$ distributed uniformly on the interval
$[0,1]$.  In the version considered here the \textsc{Quicksort} algorithm applied to the list $U_1,\ldots,U_n$ proceeds
as follows.  It places $U_1$, the first element of the list, at the root of a binary tree and divides the remaining
elements into two sublists: The elements that are smaller than $U_1$ are collected into a sublist located to the left of
$U_1$, whereas the elements larger than $U_1$ are put into a sublist located to the right of $U_1$.  (Hence
the first element of the list serves as the pivot, that is, the element used to subdivide the list).  The
procedure is then applied recursively to both sublists until only sublists of size $1$ remain.  The random tree which is
created in this way is called the \textit{binary search tree} (BST); a more detailed description will be provided in
Section~\ref{subsec:BST}.

For the analysis of the complexity of \textsc{Quicksort} the number $K_n$ of comparisons needed to sort
the list $U_1,\ldots,U_n$ is of major interest. In terms of  the tree structure of sublists this is the sum
of the depths of the nodes (also called the internal path length) of the binary search tree. As shown
by~\citet{regnier}, a suitable rescaling of $K_n$ leads to a martingale $Z_n$ that converges almost surely
to some limit variable $Z_\infty$ as $n\to\infty$,
\begin{equation}\label{eq:regnier}
     Z_n:=\frac{K_n - \E K_n}{n+1} \toas Z_{\infty}.
\end{equation}
The law $\cL(Z_\infty)$ of the limit  is  known as the \textsc{Quicksort} distribution; it has been
characterized in terms of a stochastic fixed point equation by~\citet{roesler}.

Very recently \citet{neininger} obtained a central limit theorem (CLT) accompanying~\eqref{eq:regnier} by proving
the distributional convergence
\begin{equation}\label{eq:QS2}
\sqrt{\frac{n}{2\log n}} (Z_{\infty}-Z_n) \todistr \Normal_{0,1},
\end{equation}
where $\Normal_{0,1}$ is the standard normal distribution. Neininger used the contraction method, which in the
present context has been introduced by~\citet{roesler} in connection with the distributional convergence
in~\eqref{eq:regnier}. A proof based on the method of moments followed shortly~\cite{fuchs}.

The result~\eqref{eq:QS2} is surprising as for many martingales the step from a
strong convergence result to a second-order distributional limit theorem  leads to a \textit{variance mixture} of normal
distributions; see~\citet{hall_heyde_book}. Quite generally, whenever one has a martingale convergence
result $Z_n\toas Z_{\infty}$ it is natural to ask whether there is a corresponding distributional limit theorem in the sense that, for some
normalizing sequence $b_n\to \infty$ and some non-degenerate random variable $Y$,
\begin{equation}\label{eq:CLT_martingale_intro}
b_n (Z_{\infty}-Z_{n}) \todistr Y.
\end{equation}
Indeed, provided that appropriate technical conditions (which can be found in the references cited below) are satisfied,
a distributional limit theorem of the type~\eqref{eq:CLT_martingale_intro} is known to hold if
\begin{enumerate}
\item[(a)] $Z_n$ is the proportion of black balls in the P\'olya urn after $n$ draws; see~\citet[pp.~80--81]{hall_heyde_book}.
\item[(b)] $Z_n=\sum_{i=1}^{n} a_i \xi_i$, where $\xi_1,\xi_2,\ldots$ are i.i.d.\ random variables with zero mean,
                 unit variance, and $a_1,a_2,\ldots$ is an appropriate square summable deterministic sequence; see~\citet{loynes}.
\item[(c)] $Z_n=N_n/m^n$, where $N_n$ is a supecritical Galton--Watson process with mean number of offspring $m$;
                  see~\citet{athreya} and~\citet{heyde}.
\item[(d)] $Z_n$ is the Biggins martingale of the branching random walk; see~\citet{roesler_topchii_vatutin1}.
\end{enumerate}
In this list, (a), (c) and (d) can be related to the analysis of \textsc{Quicksort}, and in all three cases, the limit
distribution is a nondegenerate mixture of normals.

We will use the well-known connection between the BST algorithm and the continuous-time branching random walk (BRW) to
explain the degeneracy phenomenon. The state at time $t$ of a BRW is a random point measure $\pi_t$ recording the
particle positions at that time; see Section~\ref{sec:BRW} for a detailed description.  A specific choice of branching
mechanism and shift distribution leads to a representation of the point measure given by the depths of the external
nodes in the BST with input size $n$ as the value $\pi_{T_n}$ at the random time
$T_n$ of the birth of the $n$th particle; see~\citet{chauvin_etal}, \cite{chauvin_rouault}, as well as the earlier work
by~\citet{devroye} that connected Galton--Watson processes and random search trees.  The BRW detour provides a new and
independent proof of Neininger's result.  In addition we obtain a stronger mode of convergence. Again, this is a topic
familiar in connection with martingale central limit theorems, where it is known that a strengthening of distributional
convergence to R\'enyi's concept of stable convergence is often possible. In our situation we can go beyond even the
stable convergence, obtaining what we call \emph{almost sure weak convergence}: With $(\cG_n)_{n\in\bN}$ the
martingale filtration we regard the conditional distribution of the left hand side
of~\eqref{eq:CLT_martingale_intro} given $\cG_n$ as a random variable with values in the set of Borel probability measures
on the real line, on this set we take the topology of weak convergence, and we show that the conditional distribution
converges almost surely in this space as $n\to\infty$. In the \textsc{Quicksort} context, with $\cG_n$ the
$\sigma$-field generated by $U_1,\ldots,U_n$, this results in
\begin{equation}\label{eq:QS3}
\cL\left\{\sqrt{\frac{n}{2\log n}} (Z_{\infty}-Z_n) \Bigg | \cG_n\right\} \toasw \{\omega\mapsto \Normal_{0,1}\}.
\end{equation}
This can be applied to obtain strong prediction intervals; see Remark~\ref{rem:prediction}.

It turns out that in our context the familiar encoding of the BRW point measures by the Biggins martingale can best be
exploited via a suitable \emph{functional} central limit theorem for the latter. The Biggins martingale arises as a
suitably standardized moment generating function of the point measures of particle positions and may thus be regarded,
together with its limit, as a stochastic process indexed by a complex parameter $\beta$ that varies over some open set
containing~$0$. For $\beta$ fixed, an associated second order distributional limit has already been obtained
by~\citet{roesler_topchii_vatutin1}, see (d) in the above list. Noting that the R\'egnier martingale appears as the
derivative at $\beta=0$ of this process we are lead to rescale $\beta$ locally in order to obtain a the functional
version that captures the local behaviour. Of course, we also want a non-trivial limit.
This is indeed possible and leads to Theorems~\ref{theo:FCLT} and~\ref{thm:mainlim}, which we regard as our main results.
Again, we obtain almost sure weak convergence, now on a suitable space of analytic functions. Further,
the distribution of the limit can be represented as the distribution of the Gaussian random analytic function given by
\begin{equation*}
  \xi(u)=\sum_{k=0}^\infty   \xi_k \frac{u^k}{\sqrt{k!}},\quad u\in\bC,
\end{equation*}
where $\xi_0,\xi_1,\ldots$ is a sequence of  independent standard normals. Much as in the
classical case of Donsker's theorem, see~\citet{billingsley_book}, this may serve as the starting point for
distributional limit theorems for various functionals of the processes, but we believe that, apart from its applicability to the
question that we started with, the BRW  functional limit theorem is of interest in its own.

Finally, the above approach is not limited to binary search trees: We also obtain an analogue of Neininger's result for
random recursive trees (RRTs). In fact, we obtain a new result even in the setting of the P\'olya urn, see
Section~\ref{subsec:polya}, and we treat Galton-Watson processes, BRW, BST, RRT with a unified method.

The paper is organized as follows. In Section~\ref{sec:BRW} we define the branching random walk and introduce the basic
notation. The functional central limit theorem for the BRW is stated in Section~\ref{sec:FCLT}. In Section~\ref{sec:asw}
we define the almost sure weak convergence and prove some of its properties.  A stronger version of the functional CLT
involving the notion of the a.s.w.\ convergence is stated in Section~\ref{sec:FCLT_strong}. In the same section, we
state a number of applications of the functional CLT including~\eqref{eq:QS2} and its analogues for other random
trees. Proofs are given in Sections~\ref{sec:moment_estimate_Biggins}, \ref{sec:proof_FCLT},
and~\ref{sec:proof_random_tree}.

\section{Branching random walk}\label{sec:BRW}

\subsection{Description of the model}
An informal picture of a \textit{branching random walk} (BRW) is that of a time-dependent random cloud of particles
located on the real line and evolving through a combination of splitting (branching) and shifting
(random walk). The particles are
replaced at the end of their possibly random lifetimes by a random number of offspring, with locations relative
to their parent being random too.
Our results will be valid for branching random walks both in discrete and continuous time. Let us describe both models.


\vspace*{2mm}
\noindent
\textit{Discrete-time branching random walk.} At time $0$ we start with one particle located at zero. At any time $n\in\N_0$ every particle which is alive at this time disappears and is replaced (independently of all other particles and of the past of the process) by a random, non-empty cluster of particles whose displacements w.r.t.\ the original particle are distributed according to some fixed point process $\zeta$ on $\R$. The number of particles in a cluster $\zeta$ is (in general) random and is always assumed to be a.s.\ finite.
Let $N_n$ be the number of particles which are alive at time $n\in \N_0$. Note that $\{N_n\colon n\in\N_0\}$ is a Galton--Watson branching process. Denote by $z_{1,n}\leq \ldots \leq z_{N_n,n}$ the positions of the particles at time $n$. Let
$$
\pi_n=\sum_{j=1}^{N_n} \delta_{z_{j,n}}
$$
be the point process recording the positions of the particles at time $n$.  The only parameter needed to identify the law of the discrete-time BRW is the law of the point process $\zeta$ encoding the shifts of the offspring particles w.r.t.\ their parent.

\vspace*{2mm}
\noindent
\textit{Continuous-time branching random walk.}
At time $0$ one particle is born at position $0$. After its birth, any particle moves (independently of all other particles and of the past of the process) according to a L\'evy process. After an exponential time with parameter $\lambda>0$, the particle disappears and at the same moment of time it is replaced by a random cluster of particles whose displacements w.r.t.\ the original particle are distributed according to some fixed point process $\zeta$. The new-born particles behave in the same way. All the random mechanisms involved are independent.  Denote the number of particles at time $t\geq 0$ by $N_t$ and  note that $\{N_t\colon t\geq 0\}$ is a branching process in continuous time. Let $z_{1,t} \leq \ldots \leq z_{N_t,t}$ be the positions of the particles at time $t$. Let
$$
\pi_t=\sum_{j=1}^{N_t} \delta_{z_{j,t}}
$$
be the point process recording the positions of the particles at time $t$. The law of the continuous-time BRW is determined by the parameters of the L\'evy process, the intensity $\lambda$, and the law of the point process $\zeta$.
\vspace*{2mm}

Both models can be treated by essentially the same methods. To simplify the notation, we will henceforth deal with the discrete-time BRW and indicate, whenever necessary, how the proofs should be modified in the continuous-time case.

\subsection{Standing assumptions and the Biggins martingale}\label{subsec:biggins_martingale}
Let us agree that  $\sum_{z\in \zeta}$ means a sum taken over all points of the point process $\zeta$, where the points are counted \textit{with multiplicities}.
We make the following \textit{standing assumptions} on the BRW.

\vspace*{2mm}
\noindent
\textsc{Assumption A:}
The cluster point process $\zeta$ is a.s.\ non-empty, finite,  and the probability that it consists of exactly one particle is strictly less than $1$.

\vspace*{2mm}
\noindent
\textsc{Assumption B:} There are $p_0>2$ and $\beta_0>0$ such that for all $\beta\in (-\beta_0,\beta_0)$,
\begin{equation}\label{eq:standing_assumption}
\E \left[\left(\sum_{z\in\pi_1} \eee^{\beta z} \right)^{p_0}\right]<\infty.
\end{equation}

\vspace*{2mm}

It follows from~\eqref{eq:standing_assumption} that the function
\begin{equation}\label{eq:m_beta_def}
m(\beta) = \E \left[ \sum_{z\in \pi_1} \eee^{\beta z}\right]
\end{equation}
is well-defined and analytic in the strip $\{\beta\in \C\colon |\Re \beta|<\beta_0\}$. Note that $m(\beta)$ is the moment generating function of the intensity measure of $\pi_1$.
Assumption~A implies that the BRW under consideration is \textit{supercritical}, that is the mean number of particles at time $1$ satisfies
$$
m:=m(0)>1.
$$
In a sufficiently small neighborhood of $0$ the function
\begin{equation}\label{eq:varphi_def}
\varphi(\beta) =\log m(\beta)
\end{equation}
is well-defined and analytic, and
the restriction of $\varphi$ to real $\beta$ is convex. By the martingale convergence theorem, there is a random variable $N_{\infty}$ such that
\begin{equation}\label{eq:N_infty_def}
\frac{N_n}{m^n} \toas N_{\infty}.
\end{equation}
Since $\E N_1^2<\infty$ (by Assumption~B) and the BRW never dies out (by Assumption~A), we have $N_{\infty}>0$ a.s. 
The assumption that $\zeta$ is non-empty could be removed (while retaining supercriticality); all results would then hold on the survival event.


A crucial role in the study of the branching random walk is played by the \textit{Biggins martingale}:
\begin{equation}\label{eq:biggins_martingale_def}
W_{n}(\beta) = \frac1 {m(\beta)^n} \sum_{z\in \pi_n} \eee^{\beta z}.
\end{equation}
\citet{uchiyama} and~\citet{biggins_uniform} proved that if Assumption~\eqref{eq:standing_assumption} holds with some $p_0\in (1,2]$, then there is $\delta_0>0$ such that the martingale $W_n(\beta)$ is bounded in $L^p$, $0< p\leq p_0$, uniformly over all $\beta\in \C$ with $|\beta|\leq\delta_0$. Furthermore, there is a random analytic function $W_{\infty}(\beta)$ defined for $|\beta|\leq \delta_0$ such that a.s.,
\begin{equation}\label{eq:biggins_martingale_converges}
\lim_{n\to \infty} \sup_{|\beta|\leq \delta_0} |W_{\infty}(\beta) - W_n(\beta)|=0.
\end{equation}
Note that $W_n(0)=\frac{N_n}{m^n}$ and $W_{\infty}(0)=N_{\infty}$, so that~\eqref{eq:biggins_martingale_converges} contains~\eqref{eq:N_infty_def} as a special case.

\subsubsection*{Notation} We denote by $\Normal_{0,\sigma^2}$ the normal distribution with mean $0$ and variance $\sigma^2$. Given a non-negative random variable $S^2$ we denote by $\Normal_{0, S^2}$ the mixture of zero mean normal distributions with random variance given by $S^2$. Throughout the paper we will use the notation
\begin{equation}\label{eq:def_sigma_tau}
\sigma^2=\Var N_{\infty} \geq 0,
\quad
d=\varphi'(0),
\quad \tau^2 =\varphi''(0)\geq 0.
\end{equation}
A generic constant which may change from line to line is denoted by $C$.

\section{Functional Central Limit Theorem for the Biggins martingale}\label{sec:FCLT}

\subsection{Statement of the FCLT}\label{subsec:FCLT_statement}
Under suitable conditions, \citet{roesler_topchii_vatutin1} proved for real $\beta$ in a certain interval around $0$ a CLT of the form
\begin{equation}\label{eq:CLT_for_W_n}
\frac{m^{\frac 12 n}}{\sqrt{\Var W_{\infty}(\beta)}}(W_{\infty}(\beta) - W_{n}(\beta)) \todistr \Normal_{0, W_{\infty}(\beta)}.
\end{equation}
Taking here $\beta=0$ and recalling that $W_n(0) = \frac{N_n}{m^n}$ one recovers the CLT for Galton--Watson processes~\cite{athreya,heyde}:
\begin{equation}\label{eq:heyde_clt}
m^{\frac 12 n} \left(N_{\infty} - \frac{N_n}{m^n}\right) \todistr \Normal_{0,  \sigma^2 N_{\infty}}.
\end{equation}
See also~\cite[p.~53]{athreya_ney_book} (discrete time case), \cite[p.~123]{athreya_ney_book} (continuous time case), \cite[Thm.~3.1, p.~28]{asmussen_hering_book} (a statement with a stronger mode of convergence), \cite[Ch.~9.2]{luschgy_book} (statistical aspects).

We will prove  a \textit{functional} version of~\eqref{eq:CLT_for_W_n}. That is, we will consider the left-hand side of~\eqref{eq:CLT_for_W_n} as a random analytic function and prove weak convergence on a suitable function space. In order to obtain a non-degenerate limit process it will be necessary to introduce a spatial rescaling into the Biggins martingale. Namely, we consider
\begin{equation}\label{eq:D_n_u_def1}
D_n(u) = m^{\frac 12 n} \left(W_{\infty}\left(\frac{u}{\sqrt{n}}\right) -  W_{n}\left(\frac{u}{\sqrt{n}}\right)\right).
\end{equation}
We have to be explicit about the function space to which $D_n$ belongs. Given $R>0$ let $\bD_R$ (resp.,\ $\overline \bD_R$) be the open (resp.,\ closed) disk of radius $R$ centered at the origin. Denote by $\bA_R$ the set of functions which are continuous on $\overline \bD_R$  and analytic in $\bD_R$. Endowed with the supremum norm, $\bA_R$ becomes a Banach space.  Note that $\bA_R$ is a closed linear subspace of the Banach space $C(\overline \bD_R)$ of continuous functions on $\overline \bD_R$. Being closed under multiplication, $\bA_R$ is even a Banach algebra. We always consider $D_n$ as a random element with values in $\bA_R$ (which is endowed with the Borel $\sigma$-algebra generated by the topology of uniform convergence). Recall that $W_n$ and $W_{\infty}$ are well defined on the disk $\overline \bD_{\delta_0}$ for some $\delta_0>0$, so that $D_n$ is indeed well defined as an element of $\bA_R$ for $n>(R/\delta_0)^2$.   Our results remain valid for some other choices o
 f the function space, for
 example one could replace $\bA_R$ by the Hardy space $H^2(\bD_R)$. Recall that $\sigma^2=\Var N_{\infty}$ and $\tau^2=\varphi''(0)$.
\begin{theorem}\label{theo:FCLT}
Fix any $R>0$. The following convergence of random analytic functions holds weakly on the Banach space $\bA_R$:
\begin{equation}
\{D_n(u)\colon u\in \overline \bD_{R}\} \toweak \{\sigma\, \sqrt {N_{\infty}}\, \xi(\tau u)\colon u\in \overline \bD_R\},
\end{equation}
where $\xi$ is a random analytic function which is defined in Section~\ref{subsec:GAF} below, and which is independent of $N_{\infty}$.
\end{theorem}
The proof of Theorem~\ref{theo:FCLT} will be given in Section~\ref{sec:proof_FCLT}. In fact, we will prove a stronger statement (Theorem~\ref{thm:mainlim}, below) in which weak convergence is replaced by the almost sure weak convergence of conditional distributions. This mode of convergence will be studied in detail in Section~\ref{sec:asw}.

\subsection{Gaussian analytic function}\label{subsec:GAF}
The random analytic function $\xi$ appearing in Theorem~\ref{theo:FCLT} is defined as follows. Let $\xi_0,\xi_1,\ldots$ be independent real standard normal variables. Consider the random analytic function $\xi:\bC\to\bC$ defined by
\begin{equation}\label{eq:xi_def}
\xi(u) = \sum_{k=0}^{\infty} \xi_k \frac{u^k}{\sqrt{k!}}.
\end{equation}
With probability $1$, the series converges uniformly on every bounded set because $\xi_n=O(\sqrt{\log n})$ a.s. Note that for every $d\in\N$ and $u_1,\ldots,u_d\in\C$, the $2d$-dimensional real random vector $(\Re \xi(u_1), \Im \xi(u_1), \ldots, \Re \xi(u_d), \Im \xi(u_d))$ is Gaussian with zero mean. The covariance structure of the process $\xi$ is given by
$$
\E [\xi(u)\xi(v)] = \eee^{uv},
\quad
\E [\xi(u)\overline{\xi(v)}] = \eee^{u\bar{v}},
\quad
u,v\in\C.
$$
It follows that $\tilde \xi(u):=\eee^{-u^2/2} \xi(u)$, $u\in\R$, is a stationary real-valued Gaussian process with covariance function
$$
\E [\tilde \xi(u)\tilde \xi(v)]
= \eee^{-\frac 12 (u-v)^2}, \quad u,v\in\R.
$$
The spectral measure of $\tilde \xi$ is the standard normal distribution.
We can view the process $\xi$ as an analytic continuation of the process $\eee^{u^2/2} \tilde \xi(u)$, $u\in\R$, to the complex plane.

A modification of $\xi$ in which the variables $\xi_0,\xi_1,\ldots$ are independent \textit{complex} standard normal is a fascinating object called the plane Gaussian Analytic Function (GAF)~\cite{sodin_tsirelson}. A remarkable feature of the plane GAF is that its zeros form a point process whose distribution is invariant with respect to arbitrary translations and rotations  of the complex plane. The law of the zero set of $\xi$ as defined in the present paper is invariant with respect to real translations only. The function $\xi$ and its complex analogue appeared as limits of certain random partition functions; see~\cite{kabluchko_klimovsky1,kabluchko_klimovsky2}.


\section{Almost sure weak convergence of probability kernels}\label{sec:asw}
Our results are most naturally stated using the notion of \textit{almost sure weak} (a.s.w.)\ convergence of probability kernels. This mode of convergence seems especially natural when dealing with randomly growing structures.  In this section we define   a.s.w.\ convergence  and study its relation to other modes of convergence.

\subsection{Basic definitions}
Let $E$ be a complete separable metric (Polish) space endowed with the Borel $\sigma$-algebra $\cE$. Let $\cM_1(E)$ be the space of probability measures on $(E, \cE)$. The weak convergence on $\cM_1(E)$ is metrized by the L\'evy--Prokhorov metric which turns $\cM_1(E)$ into a complete separable metric space.

\subsubsection*{Probability kernels}   A \textit{(probability transition) kernel} is a random variable $Q:\Omega\to \cM_1(E)$ defined on a probability space $(\Omega, \cF, \P)$ and taking values in $\cM_1(E)$. We will write $Q(\omega)$ for the probability measure on $E$ corresponding to the outcome $\omega\in\Omega$, and $Q(\omega; B)=Q(\omega)(B)$ for the value assigned by the probability measure $Q(\omega)$ to a set $B\in\cE$. Instead of the above definition of kernels we can use the following:  A kernel from a probability space $(\Omega,\cF,\P)$ to $(E,\cE)$ is a function $Q:\Omega\times \cE\to \bR$ such that
\begin{itemize}
\item[(i)] for every  set $B\in\cE$, the map $\omega\mapsto Q(\omega;B)$ is $\cF$-Borel-measurable;
\item[(ii)] for every $\omega\in\Omega$, the map $B\mapsto Q(\omega;B)$ defines a probability measure on $(E,\cE)$.
\end{itemize}
Probability kernels are also called random probability measures on $E$.

\subsubsection*{Conditional distributions}
In this paper, kernels will mostly appear in form of a conditional distribution of a random variable given a $\sigma$-algebra. Let $X:\Omega\to E$ be a random variable defined on $(\Omega,\cF,\P)$ and taking values in a Polish space $E$. Given a  $\sigma$-algebra $\cG\subset \cF$, a kernel $Q: \Omega \to \cM_1(E)$ is called (a version of) the \textit{conditional distribution} of $X$ given $\cG$ if
\begin{itemize}
\item[(i)] $Q$ is $\cG$-measurable as a map from $\Omega$ to $\cM_1(E)$,
\item[(ii)] for all bounded Borel functions $f:E\to\R$ and all $A\in\cG$,
  \begin{equation}\label{eq:cond_distr_def}
     \int_A f(X(\omega))\, \P(d\omega)
     \; = \;
     \int_A \left(\int_{E} f(z)\, Q(\omega;dz)\right) \P(\dd\omega).
\end{equation}
\end{itemize}
In this case we use the notation $Q=\cL(X|\cG)$.


\subsubsection*{Almost sure weak convergence} A sequence $Q_1,Q_2,\ldots:\Omega\to \cM_1(E)$ of kernels defined on a common probability space $(\Omega, \cF, \P)$ is said to converge \textit{almost surely  with respect to weak convergence} \textit{(a.s.w.)}\ as $n\to\infty$ if
there exists a set $A\in\cF$ with $\P[A]=1$ such that, for all $\omega\in A$, the probability measure $Q_n(\omega)$ converges
weakly on $E$ to the probability measure $Q(\omega)$, again as $n\to \infty$.

Let us state the above definition in a slightly different (but equivalent) form. Given a bounded Borel function $f:E\to\R$ and a kernel $Q$ consider the random variable $Q^f:\Omega\to\R$ defined by
$$
Q^f:\omega \mapsto \int_E f(z)Q(\omega;\dd z).
$$
Then, a sequence of kernels  $Q_1,Q_2,\ldots:\Omega\to \cM_1(E)$ converges to a kernel $Q$ in the a.s.w.\ sense if and only if for every bounded continuous function $f:E\to\R$ we have
$$
Q_n^f \toas Q^{f}.
$$
In fact, if we know that for every bounded continuous function $f$, the random variable $Q_n^f$ converges to \textit{some} limit in the a.s.\ sense, then there is a kernel $Q$ such that $Q_n$ converges to $Q$ a.s.w.;\ see~\cite{berti_etal}.
\begin{remark}
A.s.w.\ convergence contains  a.s.\ convergence as a special case. Indeed, let $X,X_1,X_2,\ldots$ be random variables on the probability space $(\Omega,\cF,\P)$. Then, the sequence $X_n$ converges a.s.\ to the random variable $X$ if and only if the sequence of kernels $Q_n:\omega\mapsto \delta_{X_n(\omega)}$ a.s.w.\ converges to the kernel $Q:\omega\mapsto \delta_{X(\omega)}$.
\end{remark}
\begin{remark}
A.s.w.\ convergence contains weak convergence as a special case. Let $\mu, \mu_1,\mu_2,\ldots$ be probability measures on $E$. The sequence $\mu_n$ converges weakly to $\mu$ if and only if the sequence of kernels $Q_n:\omega\mapsto \mu_n$ converges a.s.w.\ to the kernel $Q:\omega\mapsto\mu$.
\end{remark}
\begin{remark}
The central limit theorem can be extended to sequences of random variables which are i.i.d.\ conditionally on some $\sigma$-algebra~\cite{grzenda_zieba}. This and some related results~\cite{nowak_zieba} fit into the framework of a.s.w.\ convergence.
\end{remark}

\subsubsection*{Stable and mixing  convergence}
The a.s.w.\ convergence is related to the stable convergence which was introduced by~\citet{renyi_sets}, \cite{renyi_stable}, \cite{renyi_revesz_mixing_variables}. We recall the definition of stable convergence referring to~\cite{aldous_eagleson} for more details and references.
A sequence of kernels $Q_1,Q_2,\ldots:\Omega\to \cM_1(E)$ converges \textit{stably} to a kernel $Q:\Omega\to \cM_1(E)$ if for every set
$A\in\cF$ and every bounded continuous function $f:E\to\R$, we have
\begin{equation}\label{eq:stable_def_more_general}
\lim_{n\to\infty} \int_A \left(\int_{E} f(z) Q_{n} (\omega; \dd z)\right) \P(\dd\omega)  = \int_A \left(\int_{E} f(z) Q (\omega; \dd z)\right) \P(\dd\omega).
\end{equation}

Of particular interest for us will be the following special case of this definition. Let $X_1,X_2,\ldots$ be a sequence of random variables defined on a probability space $(\Omega,\cF,\P)$ and taking values in a Polish space $E$.
We say that $X_n$ converges stably to a kernel $Q:\Omega\to \cM_1(E)$ if the sequence of kernels $Q_n:\omega\mapsto \delta_{X_n(\omega)}$ converges stably to $Q$. That is to say, for every set $A\in\cF$ and every bounded continuous function $f:E\to\R$, we have
\begin{equation}\label{eq:stable_def}
\lim_{n\to\infty} \int_A  f(X_n(\omega)) \P(\dd\omega)
=
\int_A \left(\int_{E} f(z) Q (\omega; \dd z)\right) \P(\dd\omega).
\end{equation}
Taking in this definition $A=\Omega$ we see that stable convergence implies weak convergence of $X_n$ to the law obtained by mixing $Q(\omega)$ over $\P(\dd \omega)$.

A special case of stable convergence is the mixing convergence. We say that $X_n$ converges to a probability distribution $\mu$ on $E$ in the \textit{mixing} sense if $X_n$ converges stably to the kernel $Q:\omega \mapsto \mu$. In this case, we write
$$
X_n\tomix \mu.
$$
By the above, mixing convergence implies weak convergence to the same limit.

Another way of expressing these definitions is the following: A sequence of random variables $X_n:\Omega\to E$ converges stably if for every event $A\in \cF$ with $\P[A]>0$
the conditional distribution of $X_n$ given $A$ converges weakly to \textit{some} probability distribution  $\mu_A$ on $E$. The limiting probability distribution  is given by
$$
\mu_A := \frac{1}{\P[A]} \E [Q\ind_A]
$$
and, in general, depends on $A$.
The limiting kernel $Q$ can be seen as the Radon--Nikodym density of the $\cM_1(E)$-valued measure $A\mapsto \P[A] \mu_A$. If the limiting distribution $\mu_A$ does not depend on the choice of $A$, then we have mixing convergence.



\subsection{An example of a.s.w.\ convergence: The P\'olya urn}\label{subsec:polya}
Consider an urn initially containing $b$ black and $r$ red balls.  In each step, draw a ball from the urn at random and
replace it together with $c$ balls of the same color. Let $B_n$ and $R_n$ be the number of black and red balls after $n$
draws and let $\cF_n$ be the $\sigma$-algebra generated by the first $n$ draws. It is well-known that
the proportion $Z_n$ of black balls after $n$ draws is a martingale w.r.t.\ to the filtration $\{\cF_n\}_{n\in\N}$ and that
\begin{equation}\label{eq:Z_infty_distr}
Z_n:=\frac{B_n}{B_n+R_n}\toas Z_{\infty}\sim \Beta\left(\frac bc, \frac rc\right).
\end{equation}
We claim that
\begin{equation}\label{eq:polya_asw}
Q_n:=\cL\left\{\sqrt n (Z_{\infty} - Z_{n})\Big | \cF_n\right\} \toasw \{\omega \mapsto \Normal_{0,S^2(\omega)}\}=:Q_{\infty},
\end{equation}
where $S^2(\omega)=Z_{\infty}(\omega)(1-Z_{\infty}(\omega))$. The kernel $Q_{\infty}$ on the right-hand side maps an outcome $\omega$ to the centered normal distribution on $\R$ with variance $S^2(\omega)$.
We will prove in Proposition~\ref{prop:asw_weak} and Remark~\ref{rem:prop:asw_weak} below that~\eqref{eq:polya_asw} implies distributional convergence to the normal mixture:
\begin{equation}\label{eq:polya_asw_1}
\sqrt n (Z_{\infty}-Z_n) \todistr  \Normal_{0,S^2}.
\end{equation}
One can establish~\eqref{eq:polya_asw_1} as a direct consequence of the de Moivre--Laplace CLT by noting that conditionally on $Z_{\infty}=p$, the results of individual draws are i.i.d.\ Bernoulli variables with parameter $p$. Of course, \eqref{eq:polya_asw_1} is well-known; see~\cite[Section~3]{heyde_CLT_LIL} or~\cite[pp.~80--81]{hall_heyde_book} (where it is deduced as a special case of the CLT for martingales), but~\eqref{eq:polya_asw} is stronger than~\eqref{eq:polya_asw_1}.

\begin{proof}[Proof of~\eqref{eq:polya_asw}]
The random variables $B_n,R_n,Z_n$ are $\cF_n$-measurable. For the conditional law of $Z_{\infty}$ given $\cF_n$ we have, recalling~\eqref{eq:Z_infty_distr},
$$
\cL(Z_{\infty}|\cF_n) \sim \Beta\left(\frac{B_n}{c}, \frac {R_n}{c}\right).
$$
So, the conditional law $Q_n$ on the left-hand side of~\eqref{eq:polya_asw} is given by the kernel
$$
Q_n: \omega \mapsto \cL\left\{\sqrt{n} \left(B_{\frac 1c B_n(\omega), \frac 1c R_n(\omega)} - \frac{B_n(\omega)}{B_n(\omega)+R_n(\omega)}\right)\right\},
$$
where $B_{\alpha,\beta}$ denotes a random variable with $\Beta(\alpha,\beta)$ distribution.

We will use the following CLT for the Beta distribution. Let $\alpha_n,\beta_n> 0$ be two sequences such that $\alpha_n,\beta_n\to +\infty$ and $\frac{\alpha_n}{\alpha_n+\beta_n}\to p\in (0,1)$, as $n\to\infty$. Then,
\begin{equation}\label{eq:Beta_distr_CLT}
U_n:=\sqrt{\alpha_n+\beta_n}\left(B_{\alpha_n, \beta_n} - \frac{\alpha_n}{\alpha_n+\beta_n}\right)\todistr \Normal_{0, p(1-p)}.
\end{equation}
The proof of~\eqref{eq:Beta_distr_CLT} is standard and proceeds as follows.  Denote by
$\Gamma_{\alpha_n},\Gamma_{\beta_n}$ independent random variables having Gamma distributions with shape parameters
$\alpha_n$ and $\beta_n$ respectively, and scale parameter $1$. Since $B_{\alpha_n,\beta_n}$ has the same distribution as $\frac{\Gamma_{\alpha_n}}{\Gamma_{\alpha_n}+ \Gamma_{\beta_n}}$, we can rewrite the left-hand side of~\eqref{eq:Beta_distr_CLT} as follows:
$$
U_n\eqdistr \frac{\beta_n\Gamma_{\alpha_n} - \alpha_n \Gamma_{\beta_n}}{\sqrt{\alpha_n\beta_n (\alpha_n+\beta_n)}} \cdot \frac{\sqrt{\alpha_n\beta_n}}{\Gamma_{\alpha_n} + \Gamma_{\beta_n}}.
$$
The first factor converges weakly to the standard normal distribution (as one can easily see by computing its characteristic function), whereas the second factor converges in probability to $1$.
Slutsky's lemma completes the proof of~\eqref{eq:Beta_distr_CLT}.

Now, we apply~\eqref{eq:Beta_distr_CLT} to $\alpha_n =\frac 1c B_n(\omega)$ and $\beta_n=\frac 1c R_n(\omega)$. Noting that for a.a.\ $\omega\in\Omega$, we have $p(\omega):=\lim_{n\to\infty}\frac{\alpha_n}{\alpha_n+\beta_n} = Z_{\infty}(\omega)$ and $\alpha_n+\beta_n\sim n$, we obtain that $Q_n(\omega)$ converges weakly to $\Normal_{0, S^2(\omega)}$, for a.a.\ $\omega\in \Omega$. \end{proof}

\subsection{Properties of the a.s.w.\ convergence}
Taken together, the following proposition and examples show that a.s.w.\ convergence is strictly stronger than stable convergence.
\begin{proposition}\label{prop:asw_stable}
Let $Q_1,Q_2,\ldots:\Omega\to\cM_1(E)$ be a sequence of kernels converging to a kernel $Q:\Omega\to\cM_1(E)$ in the a.s.w.\ sense. Then, $Q_n$ converges to $Q$ stably.
\end{proposition}
\begin{proof}
Let $f:E\to\R$ be a bounded continuous function.  By definition of the a.s.w.\ convergence, the sequence $Q_n^f(\omega)=\int_{E} f(z) Q_{n} (\omega; \dd z)$ converges to $Q^f(\omega)=\int_{E} f(z) Q (\omega; \dd z)$ for a.a.\ $\omega\in\Omega$. Also, $Q_n^f(\omega)$ is bounded by $\|f\|_{\infty}$. By the dominated convergence theorem, \eqref{eq:stable_def_more_general} holds. So, $Q_n$ converges to $Q$ stably.
\end{proof}

\begin{example}
Let us show that, in general, stable convergence does not imply a.s.w.\ convergence. Let $\xi_1,\xi_2,\ldots$ be non-degenerate i.i.d.\ random variables with probability distribution $\mu$. Then, the sequence of kernels $Q_n:\omega\mapsto \delta_{\xi_n(\omega)}$ converges stably (in fact, mixing) to the kernel $Q:\omega\mapsto \mu$.  This is equivalent to saying that the i.i.d.\ sequence $\xi_1,\xi_2,\ldots$ is mixing in the sense of ergodic theory.  Alternatively, note that by the i.i.d.\ property, $\lim_{n\to\infty} \P[\xi_n\leq x | \xi_k\leq x] =\P[\xi_1\leq x]$ for every fixed $k\in\N$, and apply~\cite[Thm.~2]{renyi_sets}.  However, $Q_n$ does not converge a.s.w.\ because the sequence $\xi_n$ does not converge a.s.
\end{example}
Many classical distributional limit theorems hold, in fact, even in the sense of mixing convergence~\cite{renyi_sets,renyi_revesz_mixing_variables}. In particular, this is the case for the central limit theorem.
\begin{example}
Let $\xi_1,\xi_2,\ldots$ be i.i.d.\
random variables with $\E \xi_i=0$, $\Var \xi_i=1$. Consider the random variables $X_n=\frac 1{\sqrt n}
(\xi_1+\ldots+\xi_n)$. Then, the kernels $Q_n:\omega\mapsto \delta_{X_n(\omega)}$ converge stably (in fact, mixing) to
the kernel $Q:\omega\mapsto \Normal_{0,1}$; see~\cite[Thm.~4]{renyi_sets} or~\cite[Thm.~2]{aldous_eagleson}.  However,
$Q_n$ does not converge a.s.w.\ because the sequence $X_n$ does not converge a.s. On the other hand, the central limit
theorems for branching random walks which we will state and prove below hold not only stably but even in the a.s.w.\ sense.
\end{example}

\begin{proposition}\label{prop:asw_weak}
Let $\{\cF_n\}_{n\in\N}$ be a filtration on a probability space $(\Omega,\cF,\P)$.
Let $X_1,X_2,\ldots$ be a sequence of random variables defined on $(\Omega,\cF,\P)$ and taking values in a Polish space $E$. Assume  that for every $n\in\N$, the random variable $X_n$ is measurable w.r.t.\ the $\sigma$-algebra $\cF_{\infty}=\bigvee_{k\in\N} \cF_k$ (but not necessarily w.r.t.\ $\cF_n$).
If the sequence of conditional laws $Q_n=\cL\{X_n|\cF_n\}$ converges to a kernel $Q:\Omega\to \cM_1(E)$ in the a.s.w.\ sense, then $X_n$ converges stably to $Q$.
\end{proposition}
\begin{remark}\label{rem:prop:asw_weak}
In particular, $X_n$ converges in distribution to the probability measure $\E Q$ obtained by mixing the probability measures $Q(\omega)$ over $\P(\dd \omega)$. That is, for every Borel set $B\subset E$,
$$
(\E Q) (B) = \int_{\Omega} Q(\omega; B) \P(\dd\omega).
$$
\end{remark}
\begin{proof}[Proof of Proposition~\ref{prop:asw_weak}]
Let $f:E\to\R$ be a bounded continuous function. We will show that for every bounded $\cF$-measurable function $g:\Omega\to \R$,
\begin{equation}\label{eq:asw_stable_proof}
\lim_{n\to\infty} \int_{\Omega} f(X_n(\omega)) g(\omega) \P(\dd \omega)
=
\int_{\Omega} g(\omega) \left(\int_E f(z) Q(\omega; \dd z)\right) \P(\dd \omega).
\end{equation}
By taking $g=\ind_A$ in~\eqref{eq:asw_stable_proof} we obtain the required relation~\eqref{eq:stable_def}.

Let first $g=\ind_A$ for some $A\in \cF_k$, where $k\in \N$ is fixed. Because of the filtration property, $A\in \cF_n$ for all $n\geq k$.  Applying~\eqref{eq:cond_distr_def}  to the conditional law $Q_n=\cL(X_n|\cF_n)$, we obtain that for all $n\geq k$,
$$
\int_{A} f(X_n(\omega)) \P(\dd \omega) = \int_{A} \left(\int_{E} f(z)Q_{n} (\omega; \dd z)\right)\P(\dd \omega).
$$
For a.a.\ $\omega\in\Omega$ the probability measure $Q_n(\omega)$ converges weakly to $Q(\omega)$, and hence, the sequence $Q_n^f(\omega)=\int_{E}f(z) Q_n(\omega;\dd z)$ (which is bounded by $\|f\|_{\infty}$) converges as $n\to\infty$ to $Q^f(\omega)=\int_E f(z)Q(\omega; \dd z)$. By the dominated convergence theorem we immediately obtain~\eqref{eq:asw_stable_proof}.

A standard approximation argument  extends~\eqref{eq:asw_stable_proof} to all
$\cF_{\infty}$-measurable bounded functions $g:\Omega\to\R$.
Finally, let $g$ be $\cF$-measurable and bounded. In this case, one can  reduce~\eqref{eq:asw_stable_proof} to the case of $\cF_{\infty}$-measurable function $\tilde g=\E [g|\cF_{\infty}]$. Namely, since $X_n$ is $\cF_{\infty}$-measurable, we have
$$
\int_{\Omega} f(X_n(\omega)) g(\omega) \P(\dd \omega)
=
\int_{\Omega} f(X_n(\omega)) \tilde g(\omega) \P(\dd \omega),
$$
Similarly, since the $\cM_1(E)$-valued map $\omega\mapsto Q(\omega)$ is $\cF_{\infty}$-measurable (as an a.s.\ limit of $\cF_{\infty}$-measurable maps $\omega\mapsto Q_n(\omega)$),
$$
\int_{\Omega} g(\omega) \left(\int_E f(z) Q(\omega; \dd z)\right) \P(\dd \omega)
=
\int_{\Omega} \tilde g(\omega) \left(\int_E f(z) Q(\omega; \dd z)\right) \P(\dd \omega).
$$
So, it suffices to establish~\eqref{eq:asw_stable_proof} for the function $\tilde g$ instead of $g$, but this was already done above since $\tilde g$ is $\cF_{\infty}$-measurable and bounded.
\end{proof}
We will need the following variant of the martingale convergence theorem; see~\cite[p.~409, 10d]{loeve_book}.
An even more general result can be found in~\cite{landers_rogge}.
\begin{lemma}\label{lem:X_n_cond_F_n}
Let $\{\cF_n\}_{n\in\N}$ be a filtration on a probability space $(\Omega,\cF,\P)$. Write $\cF_{\infty} = \bigvee_{k\in\N} \cF_k$.
Let $\xi,\xi_1,\xi_2,\ldots$ be random variables defined on $(\Omega,\cF,\P)$
such that $\xi_n\to \xi$ a.s.\ and $|\xi_n|<M$ for some constant $M$. Then,
$$
\E [\xi_n | \cF_n] \toas \E[\xi|\cF_{\infty}].
$$
\end{lemma}

\begin{proposition}\label{prop:asw_properties}
Let $\{\cF_n\}_{n\in\N}$ be a filtration on a probability space $(\Omega,\cF,\P)$.
Let $X_n,Y_n$, $n\in\N$, be complex-valued random variables defined on $(\Omega,\cF,\P)$. Suppose that for some kernel $Q:\Omega\to \cM_1(\R)$,
\begin{equation}\label{eq:prop:asw_properties}
\cL(X_n|\cF_n)\toasw Q.
\end{equation}
\begin{enumerate}
\item[(a)]  If  $Y_n\to 0$ a.s., then $\cL(X_n+Y_n|\cF_n)$  converges to $Q$ a.s.w.
\item[(b)]  If  $Y_n\to 1$ a.s., then $\cL(X_nY_n|\cF_n)$  converges to $Q$ a.s.w.
\end{enumerate}
\end{proposition}
\begin{remark}
Note that we do not assume $Y_n$ to be $\cF_n$-measurable. With this assumption, the proposition would become trivial.
\end{remark}
\begin{proof}[Proof of part (a)]
We can find a sequence of uniformly continuous, bounded functions $f_1,f_2,\ldots:\R\to\R$ with the property that a sequence of probability measures $\mu_1,\mu_2,\ldots$ converges weakly on $\R$ to a probability measure $\mu$ if and only if  for every $i\in\N$,
$$
\lim_{n\to\infty} \int_{\R} f_i \dd \mu_n = \int_{\R} f_i \dd \mu.
$$
Fix some $i\in\N$. We know from~\eqref{eq:prop:asw_properties} that
\begin{equation}\label{eq:tech1}
\E [f_i (X_n)| \cF_n] \toas Q^{f_i},
\end{equation}
where $Q^{f_i}$ denotes the random variable $\omega\mapsto \int_{\R} f_i(z) Q(\omega; \dd z)$. Since $f_i$ is uniformly continuous and $Y_n\to 0$ a.s.,\ we have
$$
\xi_n := f_i(X_n+Y_n)- f_i(X_n) \toas 0.
$$
Also, $|\xi_n|\leq 2 \|f_i\|_{\infty}$. By Lemma~\ref{lem:X_n_cond_F_n} with $\xi=0$, we have $\E [\xi_n|\cF_n] \to 0$ a.s. and hence, recalling~\eqref{eq:tech1},
$$
\E [f_i (X_n+Y_n)| \cF_n] \toas Q^{f_i}.
$$
This holds for every $i\in\N$. Hence,  $\cL(X_n+Y_n|\cF_n)$ converges a.s.w.\ to $Q$.

\vspace*{2mm}
\noindent
\textit{Proof of part (b)}. Part (b) can be reduced to part (a) by noting that $X_nY_n = X_n+ X_n(Y_n-1)$ and $Y_n':=X_n(Y_n-1)$ converges a.s.\ to $0$.
\end{proof}

The following result shows that a.s.w.\ convergence of conditional laws is preserved under filtration coarsening.
\begin{proposition}\label{prop:asw_total_expectation}
Let $\{\cF_n\}_{n\in\N}$ be a filtration on a probability space $(\Omega,\cF,\P)$.
Let $\xi_1,\xi_2,\ldots$ be random variables defined on $(\Omega,\cF,\P)$ and taking values in a Polish space $E$. 
Suppose that the sequence of conditional laws $Q_n:= \cL(\xi_n|\cF_n)$ converges as $n\to\infty$ to the kernel $Q$ in the a.s.w.\ sense. Let $\{\tilde \cF_n\}_{n\in\N}$ be another filtration on $(\Omega,\cF,\P)$ such that $\tilde \cF_n\subset \cF_n$ and let $\tilde \cF_{\infty}=\bigvee_{n=1}^{\infty} \tilde \cF_n$. Then,
$$
\tilde Q_n := \cL(\xi_n | \tilde \cF_n) \toasw \E [Q | \tilde \cF_{\infty}].
$$
\end{proposition}
\begin{proof}
Let $f_1,f_2,\ldots:E\to \R$ be bounded continuous functions such that a sequence of probability measures $\mu_1,\mu_2,\ldots$ on $E$ converges weakly to $\mu$ if and only if $\int_E f_i \dd \mu_n$ converges to $\int_E f_i \dd \mu$ as $n\to\infty$, for all $i\in\N$. Let $Q_n^{f_i}:\Omega\to\R$ be the function $\omega\mapsto \int_E f_i(z) Q_n(\omega;\dd z)$ and define $\tilde Q_n^{f_i}$ similarly. Then, $Q_n\to Q$ a.s.w.\ means that $Q_n^{f_i}\to Q^{f_i}$ a.s.,\ for all $i\in\N$.  Using the definition of conditional distributions, it is easy to check that $\tilde Q_n^{f_i}=\E[Q_n^{f_i}|\tilde \cF_n]$. By Lemma~\ref{lem:X_n_cond_F_n}, we have
$$
\tilde Q_n^{f_i} = \E[Q_n^{f_i}|\tilde \cF_n]\toas \E[Q^{f_i}|\tilde \cF_{\infty}].
$$
Since this holds for every $i\in\N$ we obtain that $\tilde Q_n \to  \E [Q | \tilde \cF_{\infty}]$ a.s.w.
\end{proof}


\section{Conditional Functional Central Limit Theorem and  applications to random trees}\label{sec:FCLT_strong}
\subsection{Statement of the conditional FCLT}\label{subsec:FCLT_statement_strong}
We are almost ready to state a stronger version of Theorem~\ref{theo:FCLT}. Consider a branching random walk in discrete or continuous time defined on a probability space $(\Omega,\cF,\P)$ and satisfying the assumptions of Section~\ref{subsec:biggins_martingale}.
Denote by $\cF_t=\sigma\{\pi_j\colon 0\leq j\leq t\}$ the $\sigma$-algebra generated by the BRW up to time $t\in\N_0$ (discrete-time case) or $t\geq 0$ (continuous-time case). For our applications to the analysis of algorithms we need to state a functional CLT valid over an arbitrary increasing sequence of stopping times.
Let $0\leq T_1\leq T_2\leq \ldots$ be a monotone increasing sequence of  stopping times w.r.t.\ the filtration $\{\cF_t\}$ such that a.s.,\
\begin{equation}\label{eq:lim_T_n_infty}
\lim_{n\to\infty} T_n=+\infty.
\end{equation}
In the discrete-time case we assume additionally that $T_n$ takes values in $\N_0$.
Two special cases (which make sense both for discrete and continuous time) will be of interest to us:
\begin{enumerate}
\item $T_n=n$. 
\item $T_n$ is the time at which the $n$-th particle is born.
\end{enumerate}
The second special case will be needed for the above-mentioned applications. Let $\cF_{T_n}$ be the $\sigma$-algebra generated by the branching random walk up to the stopping time $T_n$.

Fix $R>0$. Consider the following random analytic function on the disk $\overline \bD_R$:
\begin{equation}\label{eq:D_n_u_def2}
D_{T_n}(u) = m^{\frac 12 T_n} \left(W_{\infty}\left(\frac{u}{\sqrt{T_n}}\right) -  W_{T_n}\left(\frac{u}{\sqrt{T_n}}\right)\right).
\end{equation}
We will prove that the conditional distribution of $D_{T_n}$ under $\cF_{T_n}$ converges to some limiting kernel $Q_{\infty}:\Omega\to \cM_1(\bA_R)$, in the a.s.w.\ sense.  To describe the limiting kernel $Q_{\infty}$, we use the random variable $N_\infty$ from~\eqref{eq:N_infty_def} (defined on the same probability space as the branching random walk) and the random analytic function $\xi$ described in Section~\ref{subsec:GAF} ($\xi$ may be defined on a different probability space).  For $\omega\in \Omega$ we define $Q_\infty(\omega)$ to be the distribution (on $\bA_R$) of the random analytic function
$$
\Xi(\,\cdot\,;\omega): \overline\bD_R\to\bC, \quad  u\,\mapsto\, \sigma \sqrt{N_\infty(\omega)}\,  \xi\bigl(\tau u\bigr),
\quad u\in \overline \bD_R,
$$
where we recall that $\sigma^2=\Var N_{\infty}$ and $\tau^2=\varphi''(0)$. Note that the dependence of $\Xi$ on its arguments factorizes.

The following is our main result.
\begin{theorem}\label{thm:mainlim}
As $n\to\infty$, the conditional distribution $Q_n:=\cL(D_{T_n}|\cF_{T_n})$ converges to the kernel $Q_\infty$ defined above, almost surely and with respect to weak convergence:
\begin{equation}\label{eq:thm:mainlim}
\cL\left(D_{T_n}(\cdot) \Big |\cF_{T_n}\right)
\toasw
\left\{\omega \mapsto \cL\left(\Xi(\cdot;\omega)\right)\right\}.
\end{equation}
\end{theorem}
Recalling Proposition~\ref{prop:asw_weak} and Remark~\ref{rem:prop:asw_weak}, we obtain the following
\begin{corollary}\label{cor:mainlim}
The following convergence of random analytic functions holds weakly  on $\bA_R$ for every $R>0$:
\begin{equation*}
\left\{D_{T_n}(u)\colon u\in \overline \bD_{R}\right\}
\toweak
\left\{\sigma \sqrt {N_{\infty}}\, \xi\bigl(\tau u\bigr)\colon u\in \overline \bD_{R}\right\},
\end{equation*}
where $N_{\infty}$ and $\xi$ are independent.
\end{corollary}
The proof of Theorem~\ref{thm:mainlim} will be given in Section~\ref{sec:proof_FCLT}.

\begin{remark}
The function $D_{T_n}(u)$ may not be defined on the event  $A_n := \{R/\sqrt{T_n} > \delta_0\}$.
Since we do not assume that $T_n\to\infty$ uniformly, it is possible that the probability of $A_n$ is strictly positive for every $n\in\N$. On the other hand, we have $\ind_{A_n}\to 0$ a.s.\ since $T_n\to \infty$ a.s. Hence, on the event $A_n$ we can define $D_{T_n}(u)$ in an arbitrary way  (say, as $0$) and by Proposition~\ref{prop:asw_properties}, part~(a), this does not affect Theorem~\ref{thm:mainlim} and Corollary~\ref{cor:mainlim}.
\end{remark}
\begin{remark}
Theorems~\ref{theo:FCLT} and~\ref{thm:mainlim} deal with the behavior of $W_n(\beta)$ in a small neighborhood of $0$. It is possible to obtain analogues of these results in a neighborhood of an arbitrary real $\beta_*$ from an appropriate interval; however, for our applications we need only the case $\beta_*=0$.
\end{remark}


\subsection{CLT for Galton--Watson processes}
In this section we show how Theorem~\ref{thm:mainlim} can be used to rederive and generalize the classical CLT for Galton--Watson processes due to~\citet{athreya} and~\citet{heyde}.  Consider a Galton--Watson process $N_n$ starting at time $0$ with one particle. Suppose that $N_1$ has mean $m>1$, variance $\sigma^2>0$ and finite $p_0$-th moment, for some $p_0>2$. Let $\P[N_1=0]=0$ (otherwise, we have to restrict everything to the survival event).  The limit
\begin{equation}\label{eq:N_infty_def_1}
N_{\infty}:=\lim_{n\to\infty}\frac{N_n}{m^n}>0
\end{equation}
exists a.s. By considering a branching random walk in which the particles split according to $N_n$ while not moving away from $0$, we can identify $N_n/m^n$ with $W_n(\beta)$, for every $\beta\in\C$. In this setting, Theorem~\ref{thm:mainlim} takes the form
\begin{theorem}\label{theo:heyde_CLT_very_general}
For every sequence $(T_n)_{n\in\bN}$ of stopping times with $T_n \uparrow\infty$ a.s.\ as $n\to\infty$ we have
\begin{equation}\label{eq:theo:heyde_CLT_very_general}
\cL\left( \sqrt{m^{T_n}} \left(N_{\infty} - \frac{N_{T_n}}{m^{T_n}}\right)\Bigg |\cF_{T_n}\right) \toasw
\left\{\omega \mapsto \Normal_{0, \sigma^2 N_{\infty}(\omega)}\right\}.
\end{equation}
\end{theorem}
Indeed,  $f\mapsto f(0)$ is a continuous map from $\bA_R$ to $\C$. Observe also that $\xi(0)\sim \Normal_{0,1}$ by~\eqref{eq:xi_def}.
The continuous mapping theorem justifies taking $u=0$ in Theorem~\ref{thm:mainlim} and yields~\eqref{eq:theo:heyde_CLT_very_general}.

One may ask whether it is possible to move $N_{\infty}(\omega)$ from the right-hand side of~\eqref{eq:theo:heyde_CLT_very_general} to the left. This would have the advantage that the limiting distribution would be normal rather than a mixture of normals. The question is non-trivial because the random variable $N_{\infty}$ is not $\cF_{T_n}$-measurable. Nevertheless, the answer is positive:
\begin{theorem}\label{theo:heyde_CLT_very_general_1}
For every sequence $(T_n)_{n\in\bN}$ of stopping times with $T_n \uparrow\infty$ a.s.\ as $n\to\infty$ we have
\begin{equation}\label{eq:theo:heyde_CLT_very_general_1}
\cL\left( \sqrt{\frac{m^{T_n}}{N_{\infty}}} \left(N_{\infty} - \frac{N_{T_n}}{m^{T_n}}\right)\Bigg |\cF_{T_n}\right) \toasw
\left\{\omega \mapsto \Normal_{0, \sigma^2}\right\}.
\end{equation}
\end{theorem}
\begin{proof}
Note that by~\eqref{eq:N_infty_def_1} and~\eqref{eq:lim_T_n_infty},
\begin{equation}\label{eq:Y_n_tech1}
\sqrt{\frac{m^{T_n}}{N_{T_n}}} \toas \frac 1{\sqrt{N_{\infty}}}.
\end{equation}
The random variable on the right-hand side is $\cF_{T_n}$-measurable. Applying Slutsky's lemma pointwise to Theorem~\ref{theo:heyde_CLT_very_general} we obtain that
$$
\cL\left(\sqrt{\frac{m^{T_n}}{N_{T_n}}} \sqrt{m^{T_n}} \left(N_{\infty} - \frac{N_{T_n}}{m^{T_n}}\right)\Bigg |\cF_{T_n}\right) \toasw
\left\{\omega \mapsto \Normal_{0, \sigma^2}\right\}.
$$
By Proposition~\ref{prop:asw_properties}\,(b) we can multiply the random variable on the left-hand side by $Y_n:=\sqrt{N_{T_n}/(m^{T_n}N_{\infty})}$ because $Y_n$ converges to $1$ a.s.\ by~\eqref{eq:Y_n_tech1}.  This yields~\eqref{eq:theo:heyde_CLT_very_general_1}.
\end{proof}

By Proposition~\ref{prop:asw_weak} and Remark~\ref{rem:prop:asw_weak} we obtain the following corollary of Theorems~\ref{theo:heyde_CLT_very_general} and~\ref{theo:heyde_CLT_very_general_1}.
\begin{corollary}\label{cor:heyde_clt}
It holds that
\begin{align}
&\sqrt{m^{T_n}} \left(N_{\infty} - \frac{N_{T_n}}{m^{T_n}}\right) \todistr \Normal_{0, \sigma^2 N_{\infty}},\label{eq:heyde_clt1}\\
&\sqrt{\frac{m^{T_n}}{N_{\infty}}} \left(N_{\infty} - \frac{N_{T_n}}{m^{T_n}}\right) \tomix \Normal_{0, \sigma^2}.\label{eq:heyde_clt2}
\end{align}
\end{corollary}
Taking $T_n=n$ we recover the original CLT for Galton--Watson processes; see~\eqref{eq:heyde_clt}.
Note that we need the condition $\E N_1^{p_0}<\infty$ for some $p_0>2$ (which is slightly stronger than the condition $\E N_1^2<\infty$ needed in the CLT for Galton--Watson processes).  This is due to the fact for general $T_n$'s we need to use Lyapunov's CLT in the proof of Theorem~\ref{thm:mainlim}.




\subsection{Sum of the particle positions in the BRW: Martingale convergence}
In this and the next section we will be interested in the sum of the positions of the particles in a branching random walk at time $n$:
\begin{equation}\label{eq:def_S_n}
S_n 
= \sum_{j=1}^{N_n} z_{j,n}.
\end{equation}
Let $d = \varphi'(0)$.  The sum $S_n$ is related to the first derivative $W_n'(0)$ via
\begin{equation}
L_n := W_n'(0) = \frac{S_n - dnN_n}{m^n}.
\end{equation}
>From the martingale property of $W_n(\beta)$ it follows that $L_n=W'_n(0)$ is a martingale as well.
\begin{proposition}\label{prop:L_n_conv_L_infty}
%
%
The limit $L_{\infty}:=W_{\infty}'(0)=\lim_{n\to\infty} L_n$ exists a.s.\ and in $L^p$ for every $0<p\leq p_0$.
\end{proposition}
\begin{proof}
Recall from~Section~\ref{subsec:biggins_martingale} that $W_n$, considered as a random element taking values in the Banach space $\bA_{\delta_0}$, converges a.s.\ to $W_{\infty}$, as $n\to\infty$. The mapping $f\mapsto f'(0)$ is continuous from $\bA_{\delta_0}$ to $\C$ by the Cauchy integral formula. Hence, $L_n=W'_n(0)$ converges to $L_{\infty}=W_{\infty}'(0)$ in the  a.s.\ sense.  

The proof of the $L^p$-convergence is based on a moment estimate for $W_n(\beta)$ stated in Proposition~\ref{prop:Biggins_Martingale_Lp_bounded} below.  It suffices to show that the martingale $L_n=W_n'(0)$ is bounded in $L^{p_0}$. By the Cauchy integral formula, for any sufficiently small $r>0$ we have
$$
\E |W_n'(0)|^{p_0} = \E \left|\frac {1} {2\pi} \int_0^{2\pi} \frac{W_n(r\eee^{i\varphi})}{r\eee^{i\varphi}}\dd \varphi\right|^{p_0}
\leq C \E \int_0^{2\pi} |W_n(r\eee^{i\varphi})|^{p_0} \dd \varphi,
$$
where the last step is by Jensen's inequality. Interchanging the expectation and the integral by the Fubini theorem and applying Proposition~\ref{prop:Biggins_Martingale_Lp_bounded}, we obtain the required $L^{p_0}$-boundedness: $\E |W_n'(0)|^{p_0}\leq C$.
\end{proof}
\begin{remark}
Since $\E W_n(\beta)=1$ for all $|\beta|\leq \delta_0$, we have $\E L_n = \E L_{\infty} = 0$. Consequently, $\E S_n = d n m^n$.
\end{remark}
\begin{remark}
With trivial modifications, the proof of Proposition~\ref{prop:L_n_conv_L_infty} can be extended to derivatives of arbitrary order $k\in\N_0$. Namely,  a.s.\ and in $L^p$, for every $0<p\leq p_0$, we have
\begin{equation}
W^{(k)}_n(0) \ton W_{\infty}^{(k)}(0).
\end{equation}
The $k$-th derivative $W_n^{(k)}(0)$ can be expressed through the ``empirical BRW moments''
$$
S_n^{(l)}=\sum_{j=1}^{N_n} z_{j,n}^l
$$
with $l=0,\ldots,k$. It is possible to generalize the results obtained here for $S_n=S_n^{(1)}$ to such higher moments.
\end{remark}

We will need a generalization of Proposition~\ref{prop:L_n_conv_L_infty} to arbitrary increasing sequences of stopping times. Let $0\leq T_1\leq T_2\leq \ldots$  be stopping times as in Section~\ref{subsec:FCLT_statement_strong}.
\begin{proposition}\label{prop:L_n_conv_L_infty_stop_times}
A.s.\ and in $L^p$ for every $0<p < p_0$ it holds that
\begin{equation}\label{eq:prop:L_n_conv_L_infty_stop_times}
L_{T_n} = \frac{S_{T_n} - dT_n N_{T_n}}{m^{T_n}} \ton L_{\infty}.
\end{equation}
\end{proposition}
\begin{proof}
Since $T_n\to +\infty$ a.s.,\ we have $L_{T_n}\to L_{\infty}$ a.s.\ by Proposition~\ref{prop:L_n_conv_L_infty}.
We have $|L_{T_n}|\leq \sup_{k\in\N} |L_k|$, and $L_k$ is a martingale bounded in $L^{p_0}$; see the proof of Proposition~\ref{prop:L_n_conv_L_infty}. By Doob's inequality, the sequence $L_{T_n}$ is uniformly bounded in $L^{p_0}$. By the Vitali convergence theorem, it follows that~\eqref{eq:prop:L_n_conv_L_infty_stop_times} holds in $L^p$ for all $0<p<p_0$.
\end{proof}

\begin{remark}
It remains open what moment assumption on the BRW is necessary and sufficient for Propositions~\ref{prop:L_n_conv_L_infty} and~\ref{prop:L_n_conv_L_infty_stop_times} to hold. Our standing assumption~B is certainly not the best possible. In fact, the proofs given above remain valid if we require~\eqref{eq:standing_assumption} to hold with some $p_0>1$. Anyway, in our applications to the analysis of algorithms condition~\eqref{eq:standing_assumption} is satisfied with arbitrarily large $p_0$.
\end{remark}

\subsection{Sum of the particle positions in the BRW: Conditional CLT}
Now we are ready to state a CLT for $L_{T_n}$. Let $0\leq T_1\leq T_2\leq \ldots$  be stopping times as in Section~\ref{subsec:FCLT_statement_strong}.
\begin{theorem}\label{theo:neininger_CLT_discrete}
We have
\begin{align}
\cL\left\{\sqrt{\frac{m^{T_n}}{T_n}}  \left(L_{\infty}- L_{T_n}\right)\Bigg | \cF_{T_n}\right\}
\toasw
\left\{\omega \mapsto \Normal_{0, \sigma^2\tau^2 N_{\infty}(\omega)}\right\}. \label{eq:L_n_neininger_asw}
\end{align}
\end{theorem}
\begin{proof}
Note that $f\mapsto f'(0)$ is a linear continuous map from $\bA_R$ to $\C$ by Cauchy's integral theorem; we will apply this map to both sides of~\eqref{eq:thm:mainlim}. Note that by~\eqref{eq:D_n_u_def2},
$$
D_{T_n}'(0) = \sqrt{\frac{m^{T_n}}{T_n}} (W_{\infty}'(0) - W_{T_n}'(0))=\sqrt{\frac{m^{T_n}}{T_n}} (L_{\infty} - L_{T_n}).
$$
Observe also that $\xi'(0)\sim \Normal_{0,1}$ by~\eqref{eq:xi_def}. By  the continuous mapping theorem, the a.s.w.\
convergence in~\eqref{eq:thm:mainlim} is preserved when applying the derivative map, hence we obtain~\eqref{eq:L_n_neininger_asw}.
\end{proof}

\begin{remark}
With the same justification as in Theorem~\ref{theo:heyde_CLT_very_general_1}, we can move $N_{\infty}$ from the right-hand side of~\eqref{eq:L_n_neininger_asw}  to the left-hand side.
\end{remark}

In particular, Proposition~\ref{prop:asw_weak} (see also Remark~\ref{rem:prop:asw_weak}) yields the following analogue of Corollary~\ref{cor:heyde_clt}.
\begin{corollary}
We have
\begin{align}
\sqrt{\frac{m^{T_n}}{T_n}}  (L_{\infty}- L_{T_n}) &\todistr  \Normal_{0, \sigma^2\tau^2 N_{\infty}},\label{eq:L_n_neininger_asw_cor1}\\
\sqrt{\frac{m^{T_n}}{N_{\infty}T_n}}  (L_{\infty}- L_{T_n}) &\tomix  \Normal_{0, \sigma^2\tau^2}.\label{eq:L_n_neininger_asw_cor2}
\end{align}
\end{corollary}



\subsection{Applications to random trees}\label{subsec:applications_to_trees}
In this section we show how our results can be applied to binary search trees and random recursive trees. These models
are random trees grown by attaching one new node in each step, according to certain random rules.  By randomizing the
times $T_1,T_2,\ldots$ at which the new nodes are attached, these random trees can be embedded into a suitable BRW in
continuous time; see~\citet{chauvin_etal,chauvin_rouault}. This procedure can be seen as an instance of poissonization. The embeddings are constructed such that the positions of the particles in
the BRW correspond to the depths of external (or internal) nodes of the random tree. Let $(\Omega,\cF,\P)$ be the
probability space on which the random trees are defined. The times $T_1,T_2,\ldots$ form a Yule process on some other
probability space $(\Omega',\cF',\P')$, and the BRW is then defined on the product space.
Using our results on the BRW we will obtain, after a depoissonization, results on random trees.


\vspace*{2mm}
\noindent
\textit{The Yule process.} Fix an intensity $\lambda>0$. Let $(\Omega', \cF', \P')$ be a probability space carrying independent random variables $\tau_1, \tau_2,\ldots$ with
$$
\tau_n\sim \text{Exp}(\lambda n).
$$
We regard $T_{n+1}=\tau_1+\ldots+\tau_{n}$, $n\in\N$, $T_1=0$, as times at which the $n$-th particle in a continuous-time BRW is born. We denote by $N_t=\sum_{n=1}^{\infty}\ind_{T_n\leq t}$ the number of particles at time $t\geq 0$. Then $\{N_t\colon t\geq 0\}$ is a continuous-time Markov process (called the Yule process) with values in $\N$ and transition rates
$$
n\overset{\text{intensity } \lambda n}{\xrightarrow{\hspace*{1cm}}} n+1. 
$$
One can imagine that each particle splits into two new particles with intensity $\lambda$, independently of the other
particles and of the past of the process. Note, however, that the random variables specifying \textit{which} particle splits are \textit{not} defined on the probability space $(\Omega', \cF',\P')$. The expected number of particles at time $t\geq 0$ is $\E N_t=\eee^{\lambda t}$ and hence, $m = \E N_1 = \eee^{\lambda}$. Also, it is known that
\begin{equation}\label{eq:N_infty_cont_time}
N_{\infty} =\lim_{t\to\infty} \frac{N_t}{\eee^{\lambda t}} \sim \Exp(1).
\end{equation}
In particular, in all examples below we have
$
\sigma^2=\Var N_{\infty}=1
$.

\vspace*{2mm}
\noindent
\textit{Genealogical structure and displacements.} Consider a continuous-time BRW in which the particles split at times
$T_1,T_2,\ldots$ introduced above. In any such splitting, a particle disappears and generates exactly two new
particles. We assume that the particles do not move between the splittings.  In order to specify the BRW we need to
specify the particle that splits at time $T_n$ (genealogical structure), and the displacements of its offspring. We
further assume that the random variables describing the genealogical structure and displacements are defined on a
probability space $(\Omega, \cF, \P)$. Then, the BRW can be defined on the product space $(\overline \Omega, \overline \cF, \overline
\P) = (\Omega', \cF', \P')\otimes (\Omega, \cF, \P)$. Finally, we assume that~\eqref{eq:standing_assumption} holds for
arbitrary $p_0>0$ since, as  is easy to verify,  this is true in all our examples.


Recall that we denote the positions of the particles at time $T_n$ by $z_{1,T_n}\leq\ldots\leq z_{n,T_n}$. The variable
\begin{equation}
S_{T_n}=\sum_{j=1}^n z_{j, T_n}
\end{equation}
will be interpreted below as the internal or external path length of a random tree. It is easy to see that the random variable $S_{T_n}=S_{T_n(\omega')}(\omega',\omega)$ (which is defined on the product space $\overline \Omega = \Omega'\times \Omega$) depends on the second coordinate $\omega$ only. So, we can consider $S_{T_n}$ as a random variable defined on $\Omega$.
The next theorem (whose proof we defer to Section~\ref{subsec:prop:neininger_BRW_LLN_proof}) differs from
Proposition~\ref{prop:L_n_conv_L_infty_stop_times} by a more convenient choice of normalization.
\begin{theorem}\label{prop:neininger_BRW_LLN}
Under the assumptions of the present section, on the probability space $(\Omega, \cF, \P)$ we have
\begin{equation}\label{eq:neininger_general_brw_asLp}
\tilde L_{T_n} := \frac{S_{T_n}-\frac{d}{\lambda}n\log n} {n} \ton
\tilde L_{\infty} 
\end{equation}
a.s.\ and in $L^p$ for every $p > 0$,  where
\begin{equation}\label{eq:tilde_L_infty}
\tilde L_{\infty} = \frac{L_{\infty}}{N_{\infty}} - \frac{d}{\lambda} \log N_{\infty}.
\end{equation}
\end{theorem}
\begin{remark}\label{rem:tilde_L_infty}
In the proof of Theorem~\ref{prop:neininger_BRW_LLN} we will see that the random variable $\tilde L_{\infty}$ (defined originally on the product space $\overline \Omega=\Omega'\times \Omega$) depends only on the second component $\omega\in\Omega$. By discarding the first component we can consider $\tilde L_\infty$ as a random variable on $\Omega$.
\end{remark}

The following central limit theorem is an analogue of Theorem~\ref{theo:neininger_CLT_discrete}. The proof will be given in  Section~\ref{subsec:prop:neininger_BRW_proof}. First, we need to introduce several  $\sigma$-algebras. Let $\cF_n'\subset \cF'$ be the $\sigma$-algebra on $\Omega'$ generated by $T_1,\ldots,T_n$. This $\sigma$-algebra contains information about the birth times of the particles, but it does not contain information on the genealogical and spatial structure of the BRW. Denote by $\cG_n\subset \cF$ the $\sigma$-algebra on $\Omega$ containing the information about the genealogical  structure and the displacements of the first $n$ particles in the BRW. Recall that $\cF_{T_n}\subset \cF'\otimes \cF$ is the $\sigma$-algebra on $\overline\Omega = \Omega'\times \Omega$ generated by the BRW up to time $T_n$.  Clearly, $\cF_{T_n} = \cF_n'\otimes \cG_n$.
\begin{theorem}\label{prop:neininger_BRW_CLT}
Under the assumptions of the present section, on the probability space $(\Omega, \cF, \P)$ we have
\begin{equation}\label{eq:neininger_general_brw}
\cL \left\{\sqrt{\frac{\lambda n}{\log n}} \left( \tilde L_{\infty} - \frac{S_{T_n}-\frac{d}{\lambda}n\log n} {n}\right)\Bigg | \cG_n\right\}\toasw \left\{\omega \mapsto \Normal_{0, \sigma^2\tau^2}\right\}.
\end{equation}
\end{theorem}
Using Proposition~\ref{prop:asw_weak} we obtain
\begin{corollary}\label{cor:neininger_BRW_CLT}
The following convergence holds in the mixing (and hence, distributional) sense:
\begin{equation}\label{eq:cor:neininger_general_brw}
\sqrt{\frac{\lambda n}{\log n}} \left( \tilde L_{\infty} - \frac{S_{T_n}-\frac{d}{\lambda}n\log n} {n}\right)\tomix \Normal_{0, \sigma^2\tau^2}.
\end{equation}
\end{corollary}
\begin{remark}
Note that the variance of the limiting distribution is deterministic, which is in sharp contrast to Theorem~\ref{theo:neininger_CLT_discrete}. See Remark~\ref{rem:why_variance_const} for an explanation.
\end{remark}

Now we are ready to apply these results to random trees.


\subsubsection{Binary search trees}\label{subsec:BST}
This model appears for example in the analysis of the \textsc{Quicksort} algorithm. Let $\bV=\cup_{k=0}^{\infty}\{0,1\}^k$ be the set of all finite words over the alphabet $\{0,1\}$ (including the empty word $\emptyset$). One can consider $\bV$ as the set of nodes of an infinite binary tree with root $\emptyset$. Each node $(\eps_1,\ldots,\eps_k)$ of depth $k$ is connected to two nodes $(\eps_1,\ldots,\eps_k,0)$ and $(\eps_1,\ldots,\eps_k,1)$ of depth  $k+1$. A \textit{binary tree} is a non-empty finite subset $X\subset \bV$ with the property that together with every node $(\eps_1,\ldots,\eps_k)\neq \emptyset$ it contains its predecessor $(\eps_1,\ldots,\eps_{k-1})$. The \textit{external nodes} of a binary tree $X$ are those nodes $(\eps_1,\ldots,\eps_k)\in \bV\bsl X$ for which $(\eps_1,\ldots,\eps_{k-1})\in X$. It is easy to see that the number of external nodes of $X$ exceeds the number of nodes of $X$ by $1$.

Consider a growing sequence $X_1,X_2,\ldots$ of random binary trees constructed as follows. Let $X_1$ be the tree with
one node $\emptyset$. Inductively, given $X_n$ (which is a binary tree with $n$ nodes), choose uniformly at random one
of the $n+1$ external nodes of $X_n$ and attach it to the tree. Denote the tree thus constructed by $X_{n+1}$ and
proceed further in the same manner. The random tree $X_n$ is called the \textit{binary search tree} with $n$ nodes. For
more details we refer to~\citet[Ch.~6]{drmota_book}.  We will be interested in the \textit{external path length} of
$X_n$, denoted by $\EPL_n$, which is the sum of depths of all $n+1$ external nodes of $X_n$. For example, the number
$K_n$ of comparisons used by the \textsc{Quicksort} algorithm applied to a random permutation of $n$ elements has the
same distribution as $\EPL_n-2n$. Let $(\Omega,\cF,\P)$ be the probability space on which $X_1,X_2,\ldots$ are defined
and let $\cG_n\subset \cF$ be the $\sigma$-algebra generated by $X_1,\ldots,X_n$.

Let us construct an embedding of the binary search trees into a BRW. Consider a continuous-time BRW in which the particles do not move between the splittings and each particle (located, say, at $x$) splits with intensity $\lambda = 1$ into two particles located at $x+1$:
$$
\delta_x \overset{\text{intensity } 1}{\xrightarrow{\hspace*{1cm}}} 2\delta_{x+1}.
$$
The particles of the BRW correspond to the external nodes, and their positions at time $T_n$ correspond to the depths of the
external nodes in the binary search tree with $n$ nodes. Hence, $S_{T_n}$ can be interpreted as the external path length
$\EPL_n$ of the binary search tree with $n$ nodes. We have
$$
\varphi(\beta)=2\eee^{\beta}-1,
\quad
\lambda=\varphi(0)=1,
\quad
d=\varphi'(0)=2,
\quad
\tau^2=\varphi''(0)=2.
$$
From Theorem~\ref{prop:neininger_BRW_LLN} we obtain that there is a limit random variable $\EPL_{\infty}$ such that a.s.\ and in $L^p$, for all $p>0$,
\begin{equation}
\frac{\EPL_n-2n\log n}{n} \ton \EPL_{\infty}.
\end{equation}
For $p=2$, this recovers a result of~\citet{regnier}. In view of the a.s.\ convergence, convergence in $L^p$ for general $p>0$ follows from R\"osler's~\cite{roesler} result on the convergence of the respective distributions in the Wasserstein $d_p$-metric.
From Theorem~\ref{prop:neininger_BRW_CLT} we obtain that on the probability space $(\Omega,\cF,\P)$,
\begin{equation}\label{eq:BST_CLT_conditioned}
\cL\left\{\sqrt{\frac{n}{2\log n}} \left(\EPL_{\infty} - \frac{\EPL_n-2n\log n}{n}\right)\Bigg | \cG_{n}\right\} \toasw \{\omega\mapsto\Normal_{0,1}\}.
\end{equation}
In particular, we obtain the following CLT
\begin{equation}\label{eq:BST_CLT}
\sqrt{\frac{n}{2\log n}} \left(\EPL_{\infty} - \frac{\EPL_n-2n\log n}{n}\right) \tomix \Normal_{0,1}.
\end{equation}
Thus, we recovered the CLT of~\citet{neininger}, but we have a stronger (mixing as compared to weak) mode of convergence. By the properties of mixing convergence, see~\cite[Prop.~2]{aldous_eagleson}, we also have the joint convergence
\begin{equation}\label{eq:BST_CLT_joint}
\left(\sqrt{\frac{n}{2\log n}} \left(\EPL_{\infty} - \frac{\EPL_n-2n\log n}{n}\right), \EPL_{\infty}\right) \todistr (Z, \EPL_{\infty}),
\end{equation}
where $Z\sim \Normal_{0,1}$ is independent of $\EPL_{\infty}$. This is of interest, for example, in connection with the
asymptotic distribution of the ratio of the standardized path length and its limit.

\begin{remark}\label{rem:prediction}
One can use~\eqref{eq:BST_CLT_conditioned} to construct strong prediction intervals for $\EPL_\infty$. By a strong (asymptotic) prediction interval at level $1-\alpha$ for $\EPL_\infty$ we mean two sequences of random variables $\theta_n^-$ and $\theta_n^+$ defined on $(\Omega, \cF, \P)$ such that
\begin{enumerate}
\item $\theta_n^-$ and $\theta_n^+$ are measurable w.r.t.\ $\cG_n$;
\item
$\lim_{n\to\infty} \P[\theta_n^- \leq \EPL_\infty \leq \theta_n^+| \cG_n] = 1-\alpha$ a.s.
\end{enumerate}
It follows from~\eqref{eq:BST_CLT_conditioned} that a strong prediction interval for $\EPL_\infty$ is given by
$$
\theta_n^\pm = \frac{\EPL_n-2n\log n}{n} \pm \sqrt\frac{2\log n}{n} z_{1-\frac \alpha 2},
$$
where $z_{1-\frac \alpha 2}$ is the $(1-\frac \alpha 2)$-quantile of the standard normal distribution.
\end{remark}


\subsubsection{Random recursive trees}
This well-known model, see~\citet[Ch.~6]{drmota_book}, is defined as follows. Consider a sequence of random trees $X_1,X_2,\ldots$ generated as follows. Each $X_n$ is a tree with $n$ nodes labelled by $1,\ldots,n$. The tree $X_1$ consists of one node (root) labelled by $1$. Inductively, given the tree $X_n$, we construct the tree $X_{n+1}$ as follows. Among the $n$ nodes of $X_n$ we choose one uniformly at random, attach to it a new direct descendant labeled  by $n+1$, and denote the resulting tree by $X_{n+1}$.  Denote by $(\Omega,\cF,\P)$ the probability space on which $X_1,X_2,\ldots$ are defined. Let $\cG_n\subset \cF$ be the $\sigma$-algebra generated by $X_1,\ldots,X_n$.

Let us interpret the depths of the  nodes of a random recursive tree in terms of a suitable BRW. Consider a continuous-time BRW in which the particles do not move between the splittings and each particle (located, say, at $x$) splits with intensity $1$ into one particle located at $x$ and one particle located at $x+1$:
$$
\delta_x \overset{\text{intensity } 1}{\xrightarrow{\hspace*{1cm}}} \delta_x + \delta_{x+1}.
$$
It is easy to see that the positions of the $n$ particles of the BRW at time $T_n$ have the same distribution as the depths of the nodes in a random recursive tree with $n$ nodes. Here, the depth means the distance to the node labelled by $1$.  The random variable $S_{T_n}$ can be interpreted as the internal path length, denoted by $\IPL_n$, of the random recursive tree with $n$ nodes. We have
$$
\varphi(\beta) = \eee^{\beta},
\quad
\lambda=\varphi(0)=1,
\quad
d=\varphi'(0)=1,
\quad
\tau^2=\varphi''(0)=1.
$$
From Theorem~\ref{prop:neininger_BRW_LLN} we obtain that there is a limit random variable $\IPL_{\infty}$ such that a.s.\ and in $L^p$ for every $p>0$,
\begin{equation}
\frac{\IPL_n - n \log n}{n} \ton \IPL_{\infty}.
\end{equation}
This recovers results of~\citet{mahmoud}, who proved a.s.\ and $L^2$-convergence; $L^p$-convergence for arbitrary $p>0$ has been shown by~\citet{dobrow_fill}, see~\citet{gruebel_mihailow} for a different approach. \citet{dobrow_fill} also obtained  a characterization of the
distribution of $\IPL_{\infty}$ in terms of a stochastic fixed-point equation, similar to R\"osler's result~\cite{roesler} for
the \textsc{Quicksort} distribution that we mentioned above.

From Theorem~\ref{prop:neininger_BRW_CLT} we obtain that on the probability space $(\Omega,\cF,\P)$,
\begin{equation}\label{eq:LRT_neininger1}
\cL\left\{
\sqrt{\frac{n}{\log n}} \left(\IPL_{\infty} - \frac{\IPL_n - n \log n}{n}\right)
\Big|\cG_{n}
\right\}
\toasw \{\omega\mapsto \Normal_{0,1}\}.
\end{equation}
In particular, we obtain an analogue of Neininger's CLT for random recursive trees:
\begin{equation}\label{eq:LRT_neininger2}
\sqrt{\frac{n}{\log n}} \left(\IPL_{\infty} - \frac{\IPL_n - n \log n}{n}\right)
\tomix \Normal_{0,1}.
\end{equation}
The results~\eqref{eq:LRT_neininger1} and~\eqref{eq:LRT_neininger2} seem to be new.
By~\cite[Prop.~2]{aldous_eagleson}, we  have the joint convergence
\begin{equation}\label{eq:LRT_neininger_joint}
\left(\sqrt{\frac{n}{\log n}} \left(\IPL_{\infty} - \frac{\IPL_n - n \log n}{n}\right), \IPL_{\infty}\right)
\todistr
(Z,\IPL_{\infty}),
\end{equation}
where $Z\sim \Normal_{0,1}$ is independent of $\IPL_{\infty}$.

\subsubsection{Trees and urns} It is well known that random trees of the type considered above are closely related to urn
models; for example, in~\citet{EGW1} the corresponding process boundaries were obtained by regarding the trees as nested P\'olya urns
of the type considered in Section~\ref{subsec:polya}.
Similarly, the process of node depth profiles of the external resp.\ internal
nodes in the case of binary search trees and random recursive trees is the same as the color
distribution process for a suitably chosen urn model with infinitely many colors: If the colors are numbered by the nonnegative
integers then we start at time $0$ with $1$ ball of color $0$ in both cases and proceed as follows. 
In the step from $n$ to $n+1$ we choose one of the then available $n+1$ balls uniformly at random; let $j$ be its color. In
the binary search tree case we then put back two balls with color $j+1$, in the recursive tree case we put back the
original ball and add one ball with color $j+1$.
Thus, our approach leads to results for a class of P\'olya type urn models with infinitely many colors.

\subsection{Conjectures: Laws of the iterated logarithm}
A central limit theorem is usually accompanied by a law of iterated logarithm (LIL). For example, the CLT for Galton--Watson processes~\cite{heyde} is accompanied by Heyde's LIL proved in~\cite{heyde_LIL}.

More generally, let a zero mean, $L^2$-bounded martingale $Z_n = \sum_{i=1}^n X_i$ be given. Denote by $Z_{\infty}$ the a.s.\ and $L^2$-limit of $Z_n$ and write $\sigma_n^2=\Var (Z_{\infty}-Z_n)\to 0$. \citet{heyde} provided sufficient conditions for the CLT  of the form
\begin{equation}\label{eq:martingale_tail_CLT}
\frac {Z_{\infty}-Z_n} {\sigma_n} \todistr \Normal_{0,S^2}.
\end{equation}
The most important of these conditions is this one: For some random variable $S^2$,
$$
\frac 1{\sigma_n^{2}} \sum_{i=n}^{\infty} X_i^2 \toprobab S^2.
$$
Under slightly stronger conditions, \citet{heyde} proved a law of the iterated logarithm of the form
\begin{equation}\label{eq:martingale_tail_LIL}
\limsup_{n\to\infty} \frac{Z_{\infty}-Z_n}{S \sqrt{2\sigma_n^2\log |\log \sigma_n|}} = 1.
\end{equation}
Comparing~\eqref{eq:martingale_tail_CLT} with~\eqref{eq:BST_CLT} suggests that in the setting of binary search trees with $Z_n$ being the R\'egnier martingale $\frac{\EPL_n-2n\log n}{n}$,  we should have $S=1$, $\sigma_n^2=\frac{2\log n}{n}$. So, in view of~\eqref{eq:martingale_tail_LIL},  it is natural to conjecture that in the setting of binary search trees the following LIL holds:
$$
\limsup_{n\to\infty} \frac{\sqrt n}{2\sqrt{\log n \log \log n}}\left(\EPL_{\infty} - \frac{\EPL_n-2n\log n}{n}\right)=1.
$$
An analogous conjecture can be stated for random recursive trees:
$$
\limsup_{n\to\infty} \frac{\sqrt {n}}{\sqrt{2\log n \log \log n}}\left(\IPL_{\infty} - \frac{\IPL_n- n\log n}{n}\right)=1.
$$
Similarly, the $\liminf$'s should be equal to $-1$.

\section{A moment estimate for the Biggins martingale} \label{sec:moment_estimate_Biggins}
The aim of this section is to prove that the Biggins martingale $W_n(\beta)$ is $L^p$-bounded uniformly in $|\beta|\leq \eps_0$, for some sufficiently small $\eps_0>0$.
\begin{proposition}\label{prop:Biggins_Martingale_Lp_bounded}
For every $0 < p \leq p_0$ there exist an $\eps_0>0$ and a constant $C>0$ such that for all $n\in\N$ and $\beta\in \overline \bD_{\eps_0}$ we have
$$
\E |W_n(\beta)|^p < C.
$$
\end{proposition}
\begin{remark} 
\citet{biggins_uniform} proved this  result for $p\in (1,2]$ using the von Bahr--Esseen inequality~\cite{von_bahr_esseen}. For the case $2\leq p\leq p_0$ we will use the Rosenthal inequality~\cite{rosenthal}. It states that for $p\geq 2$ and any independent random variables $X_1,\ldots,X_n\in L^p$ with zero mean we have
\begin{equation}\label{eq:rosenthal_ineq}
\E |X_1+\ldots+X_n|^p \leq K_p \left(\sum_{j=1}^n \E |X_j|^p + \left(\sum_{j=1}^n \E |X_j|^2\right)^{p/2}\right),
\end{equation}
where $K_p$ is a constant depending only on $p$.
\end{remark}
\begin{proof}[Proof of Proposition~\ref{prop:Biggins_Martingale_Lp_bounded}]
Let $2\leq p\leq p_0$. Decomposing the particles in the $(n+1)$-st generation of the BRW into clusters according to their predecessor $z_{j,n}$, $j=1,\ldots,N_n$, in the $n$-th generation, we obtain
$$
W_{n+1}(\beta) - W_n(\beta) =  \sum_{j=1}^{N_n} \frac{\eee^{\beta z_{j,n}}}{m(\beta)^{n}} g_{j,n}(\beta),
$$
where $g_{1,n}(\beta),g_{2,n}(\beta)\ldots$ are i.i.d.\ copies of $W_1(\beta)-1$ which are also independent of the $\sigma$-algebra $\cF_n$ generated by the first $n$ generations of the BRW.
By Jensen's inequality and~\eqref{eq:standing_assumption} we have the estimate, valid for all $\beta\in \C$ with $|\Re \beta|<\beta_0$,
\begin{equation}\label{eq:g_1_n_beta_moment}
\E |g_{1,n}(\beta)|^p \leq 2^{p-1} (1 + \E |W_1(\beta)|^p) \leq C+ C\E\left(\sum_{z\in \pi_1}\eee^{(\Re \beta) z}\right)^p \leq C.
\end{equation}
Noting that the random variables $\eee^{\beta z_{j,n}}$ and $N_n$ are $\cF_n$-measurable, $\E g_{j,n}(\beta)=0$,  and applying the Rosenthal inequality to the conditional distributions, we obtain
$$
\E \Big[|W_{n+1}(\beta) - W_n(\beta)|^p \Big|\cF_n\Big]
= \E \left[\left| \sum_{j=1}^{N_n} \frac{\eee^{\beta z_{j,n}}}{m(\beta)^{n}} g_{j,n}(\beta)\right|^p  \Big| \cF_n\right]
\leq K_p (A_n(\beta) +B_n(\beta)),
$$
where $A_n(\beta)$ and $B_n(\beta)$ are two terms (corresponding to the two sums on the right-hand side of~\eqref{eq:rosenthal_ineq}) which will be estimated below.  The term $A_n(\beta)$ is given by
$$
A_n(\beta) = \sum_{j=1}^{N_n} \frac{\eee^{p (\Re \beta) z_{j,n} }}{|m(\beta)|^{pn}} \E |g_{1,n}(\beta)|^p
\leq
C \left(\frac{m(p\Re \beta)}{|m(\beta)|^{p}}\right)^n W_n(p\Re \beta).
$$
where we used~\eqref{eq:biggins_martingale_def} and~\eqref{eq:g_1_n_beta_moment}. The term $B_n(\beta)$ is given by
$$
B_n(\beta)
=
\left(\sum_{j=1}^{N_n} \frac{\eee^{(2 \Re \beta) z_{j,n} }}{|m(\beta)|^{2n}} \E |g_{1,n}(\beta)|^2\right)^{p/2}
\leq
C \left(\frac{m(2\Re \beta)^{1/2}}{|m(\beta)|}\right)^{pn} |W_n(2\Re \beta)|^{p/2},
$$
where we again used~\eqref{eq:biggins_martingale_def} and the estimate $\E |g_{1,n}(\beta)|^2 <C$ following from~\eqref{eq:g_1_n_beta_moment}. We can choose $\eps_0>0$ so small that for all $|\beta|<\eps_0$,
$$
\frac{m(p\Re \beta)}{|m(\beta)|^{p}} <k <1,
\quad
\frac{m(2\Re \beta)^{1/2}}{|m(\beta)|} <k<1.
$$
Indeed, as $\beta\to 0$, the terms on the left-hand side converge to $m^{1-p}$ and $m^{-p/2}$ which are both smaller than $1$ by the supercriticality assumption $m>1$.
Now, we can estimate the expectation of $A_n(\beta)$  and $B_n(\beta)$ as follows:
\begin{align*}
\E [A_n(\beta)] &\leq C k^n \E W_n(p\Re \beta) =Ck^n,\\
\E [B_n(\beta)] &\leq C k^{pn}  \E |W_n(2\Re \beta)|^{p/2}
\leq
Ck^{pn},
\end{align*}
where in the last step we assumed that $p\in (2,4]$ and used the Biggins~\cite{biggins_uniform} estimate $\E |W_n(2\Re \beta)|^{p/2}<C$ valid for sufficiently small $\eps_0>0$ and all $|\beta| \leq \eps_0$. We obtain that for all $n\in\N$,
$$
\E \Big[|W_{n+1}(\beta) - W_n(\beta)|^p \Big] \leq C k^{pn},
$$
which implies the required bound $\E |W_n(\beta)|^p \leq C$ for $p\in (2,4]$.

Now, it is easy to drop the assumption on $p\leq 4$ inductively: If the statement was established for $p\in (2^{k-1}, 2^{k}]$, then one can repeat the above argument to obtain it for $p\in (2^{k}, 2^{k+1}]$.
\end{proof}
\begin{remark}
It is straightforward to state a continuous-time analogue of Proposition~\ref{prop:Biggins_Martingale_Lp_bounded}, just replace $n\in\N$ by $t\geq 0$.  The continuous-time case can be handled by considering a discrete skeleton of the process in the same way as in~\cite{biggins_uniform}.
\end{remark}

\section{Proof of the Functional Central Limit Theorem}\label{sec:proof_FCLT}
The aim of this section is to prove Theorem~\ref{thm:mainlim}. The main idea is a decomposition of $W_{\infty}(\beta) - W_{T_n}(\beta)$ stated in~\eqref{eq:basic_decomposition}, below. Similar decompositions appeared in the proof of the CLT for Galton--Watson processes and in the work of~\citet{roesler_topchii_vatutin1}.

\subsection{The basic decomposition}
Let $l\in \N_0$ be fixed. By the Markov property, the behavior of any particle after time $T_n$ depends only on the position of this particle at time $T_n$ but otherwise not on the behavior of the BRW before time $T_n$. In particular, for all $l\in\N$,
$$
m(\beta)^{T_n} W_{T_n+l}(\beta) = \sum_{j=1}^{N_{n}'} \eee^{\beta z_{j, T_n}} W_{j,T_n}^{(l)}(\beta),
$$
where $N_n':=N_{T_n}$ denotes the number of particles at time $T_n$, and $W_{j,T_n}^{(l)}(\beta)$, $j=1,\ldots,N_n'$, are i.i.d.\ random analytic functions (independent of the $\sigma$-algebra $\cF_{T_n}$) with the same distribution as $W_l(\beta)$. Note that these random analytic functions are defined on the same probability space as the BRW. Letting $l\to\infty$ while keeping $n$ fixed, we obtain
$$
m(\beta)^{T_n} W_{\infty} (\beta) = \sum_{j=1}^{N_{n}'} \eee^{\beta z_{j, T_n}} W_{j, T_n}(\beta),
$$
where $W_{j,T_n}$ is the a.s.\ limit of $W_{j,T_n}^{(l)}$ as $l\to\infty$; see~\eqref{eq:biggins_martingale_converges}.  Subtracting from both sides $m(\beta)^{T_n} W_{T_n}(\beta)$, we obtain the \textit{basic decomposition}
\begin{equation}\label{eq:basic_decomposition}
m(\beta)^{T_n} (W_{\infty} (\beta)- W_{T_n} (\beta)) = \sum_{j=1}^{N_{n}'} \eee^{\beta z_{j, T_n}} (W_{j, T_n}(\beta)-1).
\end{equation}
In the rest of the proof we exploit the fact that the summands on the right-hand side of~\eqref{eq:basic_decomposition} are  conditionally independent given the $\sigma$-algebra $\cF_{T_n}$. Essentially, we will prove that conditionally on $\cF_{T_n}$ it is possible to apply the Lyapunov CLT to these summands.

\begin{remark}\label{rem:why_variance_const}
At this point we can explain why the variance of the limiting normal distribution is random in the CLT for Galton--Watson processes~\eqref{eq:heyde_clt} and constant in Neininger's CLT~\eqref{eq:QS2}. In~\eqref{eq:heyde_clt} we observe a Galton--Watson process at the fixed time $T_n=n$, so that the number of summands in~\eqref{eq:basic_decomposition} is random, and this randomness persists in the large $n$ limit. In Neininger's CLT~\eqref{eq:QS2}, we consider a binary search tree with $n$ nodes meaning that the time $T_n$ is such that $N_n'=n$. So, the number of summands in~\eqref{eq:basic_decomposition} is deterministic and there is no reason for the limiting variance to be random.
\end{remark}

\subsection{The conditional distribution}
Recalling the formula for $D_{T_n}(u)$, see~\eqref{eq:D_n_u_def2}, we obtain the representation
\begin{equation}
D_{T_n}(u) = \sum_{j=1}^{N_{n}'} a_{j,n}(u) \left(W_{j, T_n}\left(\frac u {\sqrt {T_n}}\right)-1\right)
\end{equation}
where
\begin{equation}\label{eq:a_j_n_def}
a_{j,n}(u) =  m^{\frac 12 T_n} m\left(\frac u{\sqrt {T_n}}\right)^{-T_n}\eee^{\frac{u}{\sqrt {T_n}} z_{j, T_n}}.
\end{equation}

We regard the random analytic function $D_{T_n}$  as a random element with values in the Banach algebra $\bA_R$.  Note that the random analytic functions $a_{j,n}$ and the random variables $T_n$, $N_n'$ (``the past'') are $\cF_{T_n}$-measurable, while the random analytic functions $W_{j, T_n}$ (``the future'') are independent of $\cF_{T_n}$ by the Markov property. All these random objects are defined on the same probability space, say $(\Omega, \cF, \P)$, as the branching random walk. We will write $a_{j,n}(u;\omega)$, $T_n(\omega)$, $N_n'(\omega)$ if we want to stress the dependence of these random elements on $\omega\in\Omega$.

We are interested in the conditional distribution $\cL(D_{T_n}|\cF_{T_n})$  of $D_{T_n}$ given the $\sigma$-algebra $\cF_{T_n}$. To describe it, it will be convenient to ``decouple'' the ``future'' from the ``past'' by introducing independent random analytic functions $w_{j,n}(\cdot)$, $j\in\N$, which have the same law as $W_{j, T_n}(\cdot)-1$ (equivalently: the same law as $W_{\infty}(\cdot)-1$), but which are defined on a different probability space, say $(\Omega_*, \cF_*, \P_*)$.   With this notation, the conditional law $\cL(D_{T_n}|\cF_{T_n})$ is given by the kernel
\begin{equation}
Q_n:\Omega\to \cM_1(\bA_R),\quad  \omega \mapsto \cL_*(S_n(u; \omega)), \quad \omega\in\Omega,
\end{equation}
where $\cL_*$ denotes the law w.r.t.\ the probability measure $\P_*$, and $S_n(u;\omega)$ is a ``decoupled'' version of $D_{T_n}$ given by
\begin{equation}\label{eq:S_n_u_omega}
S_n(u;\omega):=\sum_{j=1}^{N_{n}'(\omega)}  a_{j,n}(u;\omega)\, w_{j,n}\left(\frac u{\sqrt{T_n(\omega)}}\right),
\quad u\in \bD_R.
\end{equation}
Keeping $\omega\in \Omega$ fixed,  we  regard $S_n(u;\omega)$ as a random element, defined on the probability space $(\Omega_*,\cF_*, \P_*)$ and taking values in $\bA_R$. For any fixed $\omega\in\Omega$, decomposition~\eqref{eq:S_n_u_omega} provides  a representation of $S_n(u;\omega)$ as a sum of independent (but not identically distributed) random elements defined on $(\Omega_*, \cF_*,\P_*)$.
Our aim is to show that for $\P$-a.a.\ $\omega\in \Omega_0$, $S_n(u;\omega)$ satisfies a central limit theorem in the sense that weakly on $\bA_R$,
\begin{equation}\label{eq_S_n_to_S_infty}
S_n(u; \omega) \toweak S_{\infty}(u;\omega),
\end{equation}
where the limit is defined as follows:
\begin{equation}
S_{\infty}(u;\omega) = \sigma\, \sqrt{N_\infty(\omega)}\,  \xi\bigl(\tau u\bigr).
\end{equation}
Here, $\xi$ is as in Section~\ref{subsec:GAF}.
Let $\Omega_0\subset \Omega$ be the set of all $\omega\in\Omega$ for which the conditions
\begin{align}
&\lim_{n\to\infty} T_n(\omega) =+\infty, \label{eq:lim_T_n_infty1}\\
&\lim_{n\to \infty} \sup_{|\beta|<\delta_0} |W_{\infty}(\beta) - W_{T_n}(\beta)|=0\label{eq:biggins_martingale_converges1}
\end{align}
are satisfied, cf.\ \eqref{eq:biggins_martingale_converges} and~\eqref{eq:lim_T_n_infty}.
Clearly, $\P[\Omega_0]=1$.   For the rest of the proof of Theorem~\ref{thm:mainlim}
\begin{center}
\textit{we keep $\omega\in\Omega_0$  fixed.}
\end{center}
The probability space $(\Omega_*,\cF_*,\P_*)$ is the only remaining source of randomness.
The proof of~\eqref{eq_S_n_to_S_infty} will be divided into two parts: convergence of finite-dimensional distributions (Section~\ref{subsec:fdd_conv})  and tightness (Section~\ref{subsec:tightness}).

\subsection{Convergence of finite-dimensional distributions}\label{subsec:fdd_conv}
Fix some $u_1,\ldots,u_d\in \C$. Our aim is to prove that
$$
(S_n(u_1; \omega),\ldots, S_n(u_d; \omega)) \tofd (S_{\infty}(u_1;\omega),\ldots, S_{\infty}(u_d;\omega)).
$$
This is done by verifying the conditions of the Lyapunov central limit theorem for the decomposition~\eqref{eq:S_n_u_omega}. We can treat $a_{j,n}(u;\omega)$, $N_n'(\omega)$, $T_n(\omega)$ as deterministic, while $w_{j,n}$ are considered as $\bA_R$-valued random elements defined on the probability space $(\Omega_*,\cF_*,\P_*)$.

\vspace*{2mm}
\noindent
\textit{Step 1: Convergence of covariances.} Take some $u,v\in\C$. We show that 
\begin{align}
&\lim_{n\to\infty} \E_*[S_n(u;\omega) S_n(v;\omega)] = \sigma^2\, N_{\infty}(\omega)\, \eee^{\tau^2 uv},\label{eq:conv_cov_statement_1}\\
&\lim_{n\to\infty} \E_*[S_n(u;\omega) \overline{S_n(v;\omega)}] = \sigma^2\, N_{\infty}(\omega)\, \eee^{\tau^2 u\bar {v}}. \label{eq:conv_cov_statement_2}
\end{align}
Here, $\E_*$ denotes the expectation operator w.r.t.\ the probability measure $\P_*$. We prove only~\eqref{eq:conv_cov_statement_1} since the proof of~\eqref{eq:conv_cov_statement_2} is analogous. Since $a_{j,n}(u)$ and $a_{j,n}(v)$ are deterministic, we have
$$
\E_*[S_n(u) S_n(v)] = \left(\sum_{j=1}^{N_n'} a_{j,n}(u) a_{j,n}(v)\right)  \E_*\left[w_{j,n}\left(\frac{u}{\sqrt{T_n}}\right)w_{j,n}\left(\frac{u}{\sqrt{T_n}}\right)\right]
$$
The proof of~\eqref{eq:conv_cov_statement_1} will be accomplished after we have shown that 
\begin{align}
&\lim_{n\to\infty} \sum_{j=1}^{N_n'} a_{j,n}(u) a_{j,n}(v)  = N_{\infty}\, \eee^{\tau^2
 uv},\label{eq:cov_conv_1}\\
&\lim_{n\to\infty} \E_*\left[w_{1,n}\left(\frac{u}{\sqrt{T_n}}\right)w_{1,n}\left(\frac{v}{\sqrt{T_n}}\right)\right] = \sigma^2. \label{eq:cov_conv_2}
\end{align}

\vspace*{2mm}
\noindent
\textit{Proof of~\eqref{eq:cov_conv_1}.}
Using first the definition of $a_{j,n}$, see~\eqref{eq:a_j_n_def}, and then the uniformity in~\eqref{eq:biggins_martingale_converges},  we obtain that
\begin{align*}
\sum_{j=1}^{N_n'} a_{j,n}(u) a_{j,n}(v)
&=
\eee^{T_n\left(\varphi(0) - \varphi\left(\frac u{\sqrt{T_n}}\right) - \varphi\left(\frac v{\sqrt{T_n}}\right)\right)}
\sum_{j=1}^{N_n'} \eee^{\frac{u+v}{\sqrt n} z_{j,T_n}}\\
&\sim
N_{\infty}\eee^{T_n\left(\varphi(0) - \varphi\left(\frac u{\sqrt{T_n}}\right) - \left(\frac v{\sqrt{T_n}}\right) + \varphi\left(\frac{u+v}{\sqrt{T_n}}\right)\right)}.
\end{align*}
Expanding $\varphi$ into a Taylor series at $0$, we obtain~\eqref{eq:cov_conv_1}.

\vspace*{2mm}
\noindent
\textit{Proof of~\eqref{eq:cov_conv_2}.}
Recall that $\lim_{n\to\infty} T_n = +\infty$. Since $w_{1,n}$ has the same law as $W_{\infty}-1$ and as such is continuous at $0$, we have, $\P_*$-a.e.,\
\begin{equation}
\lim_{n\to\infty} w_{1,n}\left(\frac{u}{\sqrt {T_n}}\right) w_{1,n}\left(\frac v{\sqrt{T_n}}\right) = w_{1,n}^2(0).
\end{equation}
We have to prove the uniform integrability in order to be able to conclude the convergence of expectations. By Proposition~\ref{prop:Biggins_Martingale_Lp_bounded},
\begin{equation}\label{eq:w_j_n_bounded_L_p}
\E\left|w_{1,n}\left(\frac{u}{\sqrt {T_n}}\right)\right|^{2+\delta} < C,
\quad
\E\left|w_{1,n}\left(\frac{v}{\sqrt {T_n}}\right)\right|^{2+\delta} < C,
\end{equation}
where $C=C(\omega)$ may depend on $\omega$.
By the Cauchy--Schwarz inequality, the sequence $w_{1,n}(u/\sqrt{T_n}) w_{1,n}(v/\sqrt{T_n})$  is bounded in $L^{1+\frac \delta 2}(\Omega_*,\cF_*,\P_*)$, which implies that it is  uniformly integrable.
It follows from~\eqref{eq:cov_conv_2} that
$$
\lim_{n\to\infty} \E_* \left[w_{1,n}\left(\frac{u}{\sqrt {T_n}}\right) w_{1,n}\left(\frac v{\sqrt{T_n}}\right)\right] = \E_* [w_{1,n}^2(0)] = \Var W_{\infty}(0) = \sigma^2,
$$
where in the last step we used that under $\P_*$ the random variable $w_{1,n}(0)$ has the same distribution as the random variable $W_{\infty}(0)-1=N_{\infty}-1$ under $\P$.


\vspace*{2mm}

\noindent
\textit{Step 2: Lyapunov condition.} We  verify that for every $u\in \C$, the Lyapunov condition $\lim_{n\to\infty} R_n(u)=0$ holds, where
$$
R_n(u)
= \sum_{j=1}^{N_n'} \E_*\left|a_{j,n}(u) w_{j,n}\left(\frac{u}{\sqrt{T_n}}\right)\right|^{2+\delta}.
$$
Using~\eqref{eq:w_j_n_bounded_L_p} and recalling the definition of $a_{j,n}$, see~\eqref{eq:a_j_n_def}, we obtain
\begin{align*}
R_n(u)
\leq
C \sum_{j=1}^{N_n'} \left|a_{j,n}(u)\right|^{2+\delta}
=
C \eee^{T_n\left(\frac{2+\delta}{2} \varphi(0) - (2+\delta) \varphi\left(\frac{\Re u}{\sqrt{T_n}}\right)\right)} \sum_{j=1}^{N_n'} \eee^{(2+\delta)(\Re u)\frac{z_{j,T_n}}{\sqrt{T_n}}}.
\end{align*}
Using~\eqref{eq:biggins_martingale_converges1} we obtain that uniformly in $u\in \bD_R$,
\begin{align*}
R_n(u)
\leq
C N_{\infty} \eee^{T_n\left(\frac{2+\delta}{2} \varphi(0) - (2+\delta) \varphi\left(\frac{\Re u}{\sqrt{T_n}}\right) + \varphi\left(\frac{(2+\delta)\Re u}{\sqrt{T_n}}\right)\right)}.
\end{align*}
Expanding $\varphi$ into a Taylor series at $0$, we obtain the estimate
$$
R_n(u) \leq C N_{\infty} \eee^{- \left(\frac {\delta}{2} + o(1) \right)T_n}.
$$
This completes the verification of the Lyapunov condition.

\subsection{Tightness}\label{subsec:tightness}
We prove that for every $\omega\in\Omega_0$, the sequence of random analytic functions $S_n(u;\omega)$, $n\in\N$, is tight on $\bA_R$.
\begin{lemma}\label{lem:S_n_variance_estimate}
Fix $R>0$. There exist  random variables $M:\Omega\to\R$ and $N:\Omega\to \N$ such that for all $\omega\in\Omega_0$, $n>N(\omega)$, $u\in \bD_{R}$,
\begin{equation}\label{eq:tightness_shirai}
\E_* |S_n(u;\omega)|^2 \leq M(\omega).
\end{equation}
\end{lemma}
The required tightness can be now established as follows. A result of~\citet{shirai} (see Lemma~2.6 in~\cite{shirai} and the remark thereafter)  states that if $f_1,f_2,\ldots$ are random analytic functions defined on the disk $\bD_{2R}$ such that for some $q>0$, $C>0$ and all $n\in \N$, $u\in\bD_{2R}$, we have $\E |f_n(u)|^q<C$, then the sequence $f_n$ is tight on the space of analytic functions on the smaller disk $\overline \bD_{R}$. Since Lemma~\ref{lem:S_n_variance_estimate} holds with $R$ replaced by $2R$, the result of Shirai implies that for every $\omega\in \Omega_0$, the sequence $S_n(u;\omega)$, $n\in\N$, is  tight on the space of analytic functions on the disc $\overline \bD_{R}$. 

\begin{proof}[Proof of Lemma~\ref{lem:S_n_variance_estimate}]
For every $\omega\in\Omega_0$ we have $\lim_{n\to\infty} T_n(\omega)=+\infty$ and hence, we can choose a large enough $N(\omega)$ such that for all $n>N(\omega)$ the argument of the function $w_{j,n}$ in the definition of $S_n(u;\omega)$, see~\eqref{eq:S_n_u_omega}, is small enough so that $S_{n}(u;\omega)$ is well-defined for all $u\in \bD_R$.

Fix some $\omega\in \Omega_0$ and let in the sequel $n>N(\omega)$. Note that $\E_* S_n(u;\omega)=0$. Using the additivity of the variance and~\eqref{eq:w_j_n_bounded_L_p} we obtain that for some $C_1=C_1(\omega)$ and all $n>N(\omega)$,
$$
\E_* |S_n(u)|^2
=
\sum_{j=1}^{N_n'} |a_{j,n}(u)|^2 \, \E_* \left|w_{j,n}\left(\frac{u}{\sqrt{T_n}}\right)\right|^2
\leq
C_1 \sum_{j=1}^{N_n'} |a_{j,n}(u)|^2.
$$
Recalling the definition of $a_{j,n}$, see~\eqref{eq:a_j_n_def},  and using~\eqref{eq:biggins_martingale_converges1}, we obtain that
\begin{align*}
\E_* |S_n(u)|^2
&\leq C_1 \eee^{T_n\left(\varphi(0)-2\Re \varphi\left(\frac{u}{\sqrt{T_n}}\right)\right)} \sum_{j=1}^{N_n'} \eee^{\frac{2(\Re u) z_{j,T_n}}{\sqrt{T_n}}}\\
& = C_1 \eee^{ T_n \left(\varphi(0)-2\Re \varphi\left(\frac{u}{\sqrt{T_n}}\right)+ \varphi\left(\frac{2\Re u}{\sqrt{T_n}}\right) \right)} W_{T_n}\left(\frac{2\Re u}{\sqrt{T_n}}\right).
\end{align*}
Expanding $\varphi$ into a Taylor series at $0$, we see that the argument of the exponential function can be estimated by $C_2=C_2(\omega)$. Also, for all $\omega\in\Omega_0$,
$$
\lim_{n\to\infty} W_{T_n}\left(\frac{2\Re u}{\sqrt{T_n}};\omega\right) = W_{\infty}(0;\omega),
$$
thus proving~\eqref{eq:tightness_shirai}.
\end{proof}

\section{Proofs of the random tree results}\label{sec:proof_random_tree}
This section contains depoissonization arguments justifying the passage from BRW to random trees.

\subsection{Proof of Theorem~\ref{prop:neininger_BRW_LLN}}\label{subsec:prop:neininger_BRW_LLN_proof}
Recall that
\begin{equation}\label{eq:L_T_n_recall}
L_{T_n} = \frac{S_{T_n} - dn T_n }{\eee^{\lambda T_n}},
\quad
\tilde L_{T_n} = \frac{S_{T_n}-\frac{d}{\lambda}n\log n} {n},
\quad
\tilde L_{\infty} = \frac{L_{\infty}}{N_{\infty}} - \frac{d}{\lambda} \log N_{\infty}.
\end{equation}
We are going to show that on the product probability space $(\overline \Omega, \overline \cF, \overline \P)$  it holds that $\tilde L_{T_n}\to \tilde L_{\infty}$ a.s.\ and in $L^p$ for all $p>0$.

\vspace*{2mm}
\noindent
\textit{Step 1: Proof of the a.s.\ convergence.} Let us show that $\tilde L_{T_n}\to \tilde L_{\infty}$ a.s. By~\eqref{eq:L_T_n_recall},
\begin{equation}\label{eq:tech2}
\tilde L_{T_n} = L_{T_n} \frac{\eee^{\lambda T_n}}{n} + \frac d{\lambda}\left(\lambda T_n -\log n\right).
\end{equation}
By Proposition~\ref{prop:L_n_conv_L_infty_stop_times} (in the continuous-time version) we have
\begin{equation}\label{eq:aux_L_T_n}
L_{T_n}\toas L_{\infty}.
\end{equation}
The a.s.\ convergence of the martingale $\frac{N_t}{\eee^{\lambda t}}$ to  $N_{\infty}$ as $t\to+\infty$ implies, with $t=T_n$, that
\begin{equation}
\frac{n}{\eee^{\lambda T_n}} \toas N_{\infty},
\quad
\lambda T_n =  \log n - \log N_{\infty} + o(1)
\quad \text{ a.s.} \label{eq:aux_as0}
\end{equation}
Inserting~\eqref{eq:aux_L_T_n} and~\eqref{eq:aux_as0} into~\eqref{eq:tech2} yields that $\tilde L_{T_n}\to \tilde L_{\infty}$ a.s.

Since $\tilde L_{T_n}$ depends only on $\omega\in\Omega$ (and not on $\omega'\in\Omega'$), the same is true for the limit random variable $\tilde L_\infty$. Hence,  we can regard $\tilde L_{T_n}$ and $\tilde L_\infty$ as random variables on the probability space  $(\Omega,\cF,\P)$, and the convergence $\tilde L_{T_n}\to \tilde L_\infty$ holds on this probability space as well.

\vspace*{2mm}
In the next two steps we prove that $\tilde L_{T_n}\to \tilde L_{\infty}$ in $L^p(\overline\Omega,\overline \cF,\overline\P)$ for every $p>0$. In fact, by the Vitali convergence theorem, it suffices to prove that the sequence $\tilde L_{T_n}$ is bounded in $L^p$ for every $p>0$.

\vspace*{2mm}
\noindent
\textit{Step 2: Proof that $L_{T_n}^*$ is bounded in $L^p$.}
Consider first
$$
L_{T_n}^*:=\frac{S_{T_n} - \frac{d}{\lambda} n\log n}{\eee^{\lambda T_n}}
=
L_{T_n} + \frac{d}{\lambda} \frac{n}{\eee^{\lambda T_n}} (\lambda T_n - \log n).
$$
By Proposition~\ref{prop:L_n_conv_L_infty_stop_times} we know that $L_{T_n}$ is bounded in $L^p$. By the Minkowski and H\"older inequalities it suffices to show that for some $C_p>0$ depending only on $p>0$,
\begin{equation}\label{eq:moments_T_n}
   \E \left(\frac{n}{\eee^{\lambda T_n}}\right)^p <C_p,\qquad
   \E |\lambda T_n - \log n|^p <C_p.
\end{equation}
Recall that $T_n$ is the time at which the $n$-th particle is born in a Yule process with intensity $\lambda$. This means that
$$
E_k := \lambda k (T_{k+1}-T_{k}), \quad k\in\N,
$$
are i.i.d.\ exponential random variables with parameter $1$.
We have the representation
\begin{equation}\label{eq:rep_T_n}
\lambda T_n = \sum_{k=1}^{n-1} \frac{E_k}{k}.
\end{equation}
It follows that for every $r>-1$,
\begin{equation}\label{eq:moments_T_n_1}
\E \left(\frac{n}{\eee^{\lambda T_n}}\right)^r = n^r\prod_{k=1}^{n-1} \frac{1}{1+ \frac r k} \ton \Gamma(r+1).
\end{equation}
This implies the first estimate in~\eqref{eq:moments_T_n}. Also, for any $0<\eps<1$ we have
$$
\E |\lambda T_n - \log n|^p \leq C \E \left(\frac{n}{\eee^{\lambda T_n}}\right)^{\eps} + C\E \left(\frac{n}{\eee^{\lambda T_n}}\right)^{-\eps} < C_p.
$$
This proves the second estimate in~\eqref{eq:moments_T_n}.

\vspace*{2mm}
\noindent
\textit{Step 3: Proof that $\tilde L_{T_n}$ is bounded in $L^p$.}
We proved that the sequence $L^*_{T_n}$ is bounded in $L^p$, but we need a similar statement for the sequence $\tilde L_{T_n}$. Note that the random variables $S_{T_n}$ and $T_n$ are independent.  We have, by Step~2,
$$
C_p > \E |L_{T_n}^*|^p = \E \left|\tilde L_{T_n} \frac{n}{\eee^{\lambda T_n}}\right|^p =
\E |\tilde L_{T_n}|^p \, \E \left(\frac{n}{\eee^{\lambda T_n}}\right)^p > c_p \E |\tilde L_{T_n}|^p,
$$
where $c_p>0$ and the last inequality is by~\eqref{eq:moments_T_n_1}. Hence, the sequence $\E |\tilde L_{T_n}|^p$ is bounded.
\hfill $\Box$

\subsection{Proof of Theorem~\ref{prop:neininger_BRW_CLT}}\label{subsec:prop:neininger_BRW_proof}
We have to show that on the probability space $(\Omega,\cF,\P)$,
\begin{equation}\label{eq:neininger_BRW_main}
\cL \left\{\sqrt{\frac{\lambda n}{\log n}} \left( \tilde L_{\infty} - \frac{S_{T_n}-\frac{d}{\lambda}n\log n} {n}\right)\Bigg | \cG_{n}\right\}\toasw \left\{\omega \mapsto \Normal_{0, \sigma^2\tau^2}\right\},
\end{equation}
where we recall from~\eqref{eq:tilde_L_infty} that  $\tilde L_{\infty} = \frac{L_{\infty}}{N_{\infty}} - \frac{d}{\lambda} \log N_{\infty}$. Instead, we will show that on the product space $(\overline \Omega,\overline \cF,\overline\P)$,
\begin{equation}\label{eq:neininger_BRW_main1}
\cL \left\{\sqrt{\frac{\lambda n}{\log n}} \left( \tilde L_{\infty} - \frac{S_{T_n}-\frac{d}{\lambda}n\log n} {n}\right)\Bigg | \cF_{T_n}\right\}\toasw \left\{\omega \mapsto \Normal_{0, \sigma^2\tau^2}\right\}.
\end{equation}
Assuming that we have established~\eqref{eq:neininger_BRW_main1}, let us prove~\eqref{eq:neininger_BRW_main}. Note that $\cF_{T_n} = \cF'_n \otimes \cG_n$, so that Proposition~\ref{prop:asw_total_expectation} allows us to replace $\cF_{T_n}$ in~\eqref{eq:neininger_BRW_main1} by the smaller $\sigma$-algebra $\{\emptyset,\Omega\}\otimes \cG_n$. But since the random variable on the left-hand side of~\eqref{eq:neininger_BRW_main1} (defined on the product space $\overline \Omega=\Omega'\times \Omega$) depends only on the coordinate $\omega\in\Omega$, we can discard the component $\Omega'$ and obtain~\eqref{eq:neininger_BRW_main}. In the sequel, we are occupied with the proof of~\eqref{eq:neininger_BRW_main1}.

\vspace*{2mm}
\noindent
\textit{Step 1: Proof strategy.} Recalling that $T_n$ is the time at which the $n$-th particle is born and using~\eqref{eq:aux_as0}, we can write Theorem~\ref{theo:neininger_CLT_discrete} in the following form: On the product space $(\overline \Omega,\overline \cF,\overline\P)$,
\begin{equation}\label{eq:aux_neininger}
\cL\left\{
\sqrt{\frac{\eee^{\lambda T_n}}{T_n N_{\infty}}}
\left(L_{\infty}- \frac{S_{T_n} - d n T_n }{\eee^{\lambda T_n}} \right)\Bigg | \cF_{T_n}\right\}
\toasw
\left\{\omega \mapsto \Normal_{0, \sigma^2\tau^2 }\right\}.
\end{equation}
Inserting~\eqref{eq:aux_as0} into equation~\eqref{eq:aux_neininger} formally, we obtain the required relation~\eqref{eq:neininger_BRW_main1}. However, in order to obtain~\eqref{eq:neininger_BRW_main1} rigorously we need slightly more precise asymptotics than those given in~\eqref{eq:aux_as0}.

\vspace*{2mm}
\noindent
\textit{Step 2: Precise asymptotics for $T_n$.}
We prove that
\begin{equation}\label{eq:lil_chow_teicher_1}
\limsup_{n\to \infty} \frac{\left|N_{\infty}\eee^{\lambda T_n} - n\right|}{\sqrt{2n\log\log n}}
=
\limsup_{n\to\infty} \frac{\left|\lambda T_n - \log \frac{n}{N_{\infty}}\right|}{\sqrt{2 n^{-1}\log \log n}}
=
1
\quad
\text{a.s.}
\end{equation}
We need Kendall's theorem; see~\cite[Thm.~2 on p.~127]{athreya_ney_book}. It states that conditionally on $N_{\infty}=y>0$, the points $P_{n}:= y (\eee^{\lambda T_n}-1)$, $n\geq 2$, form a homogeneous Poisson point process on $(0,\infty)$. By the law of the iterated logarithm for the Poisson process, we have
$$
\limsup_{n\to\infty} \frac{|P_n-n|}{\sqrt{2n\log\log n}} = 1.
$$
After standard transformations, we obtain~\eqref{eq:lil_chow_teicher_1}. Alternatively, the second limit in~\eqref{eq:lil_chow_teicher_1} could be computed using Heyde's~\cite{heyde_LIL} law of the iterated logarithm applied to the Yule process $N_t$ evaluated at time $t=T_n$.

\vspace*{2mm}
\noindent
\textit{Step 3: Completing the proof.}
We can represent the random variable on the left-hand side of~\eqref{eq:neininger_BRW_main1} as a sum of three terms:
\begin{multline}\label{eq:aux_dec3}
\sqrt{\frac{\lambda n}{\log n}}\left(\tilde L_{\infty} - \frac{S_{T_n}-\frac{d}{\lambda}n\log n} {n}\right)
\\
=
\sqrt{\frac{\lambda T_n}{\log n}\, \frac{\eee^{\lambda T_n} N_{\infty}}{n}}\cdot \sqrt{\frac{\eee^{\lambda T_n}}{T_n N_{\infty}}} \left(L_{\infty}- \frac{S_{T_n}-dnT_n} {\eee^{\lambda T_n}} \right)
\\+
\sqrt{\frac{\lambda n}{\log n}}L_{\infty} \left(\frac{1}{N_{\infty}} - \frac{\eee^{\lambda T_n}}{n}\right)
+
\sqrt{\frac{\lambda n}{\log n}}\frac{d}{\lambda} \left(\log\frac{n}{N_{\infty}}-\lambda T_n\right)
\end{multline}
Denote the three summands on the right-hand side of~\eqref{eq:aux_dec3} by $R_n^{(1)}, R_n^{(2)}, R_n^{(3)}$.
It follows from~\eqref{eq:lil_chow_teicher_1} and~\eqref{eq:aux_as0} that
\begin{equation}\label{eq:aux_dec3_terms_negligible}
\lim_{n\to\infty} R_n^{(2)} = \lim_{n\to\infty} R_n^{(3)} = 0 \quad \text{a.s.},
\quad
\lim_{n\to\infty} \sqrt{\frac{\lambda T_n}{\log n}\, \frac{\eee^{\lambda T_n} N_{\infty}}{n}}=1 \quad \text{a.s.}
\end{equation}
Applying to the decomposition on the right-hand side of~\eqref{eq:aux_dec3}  equations~\eqref{eq:aux_neininger} and~\eqref{eq:aux_dec3_terms_negligible} together with Proposition~\ref{prop:asw_properties}, we obtain the required equation~\eqref{eq:neininger_BRW_main1}.
\hfill $\Box$

\section*{Acknowledgment} Zakhar Kabluchko is grateful to Pascal Maillard for a discussion related to decomposition~\eqref{eq:basic_decomposition}. The authors are grateful to Henning Sulzbach and the unknown referee for pointing out a gap in the previous version of the paper and making a number of useful comments.

\bibliographystyle{plainnat}
\bibliography{clt_quicksort_bib}

\end{document}